\newif\ifdraft
\documentclass[12pt]{amsart}
\usepackage{enumitem}
\usepackage{verbatim}
\usepackage{xcolor}
\usepackage[margin=1.3in]{geometry}
\usepackage{tikz}
\usetikzlibrary{fadings}

\usepackage{hyperref}

\ifdraft
\usepackage{fancyhdr,datetime2}
\fancypagestyle{firstpage}{%
	
	\fancyhf{}
	\fancyhead[C]{\color{red}\fbox{\textbf{DRAFT} (\DTMnow) -- NOT FOR DISTRIBUTION}}
	\fancyfoot[C]{\thepage}
}
\fi

\numberwithin{equation}{section}

\usepackage{mymacros}

\begin{document}

\title{Maximizing entropy for power-free languages}
\author{Vaughn Climenhaga}
\address{Dept.\ of Mathematics, University of Houston, Houston, TX 77204}
\date{September 9, 2026}
\thanks{This material is based upon work supported by the National Science Foundation under Award No.\ DMS-2154378 and DMS-2453314.}

\begin{abstract}
A power-free language is characterized by the number of symbols used and a limit on how many times a block of symbols can repeat consecutively. For certain values of these parameters, it is known that the number of legal words grows exponentially fast with respect to length. In the terminology of dynamical systems and ergodic theory, this means that the corresponding shift space has positive topological entropy. 

We prove that in many cases, this shift space has a unique measure of maximal entropy. The proof uses a weak analogue of Bowen's specification property. The lack of any periodic points in power-free shift spaces stands in striking contrast to other applications of specification-based techniques, where the number of periodic points often has exponential growth rate given by the topological entropy.
\end{abstract}

\maketitle
\ifdraft
\thispagestyle{firstpage}
\pagestyle{fancy}
\fi

\section{Main results}%
\label{sec:rep-free}

\subsection{Introduction}

Given a natural number $d\geq 2$ and a real number $\beta>1$, the \emph{$d$-ary $\beta$-free shift space} $\Xdb$ is the set of all bi-infinite strings $x\in \{1,\dots, d\}^\ZZ$ such that no subword of $x$ is an $\alpha$-power for any rational $\alpha\geq \beta$ (see \S\ref{sec:def-intro} for precise definitions). Similarly, $\Xdbp$ consists of all $x$ that avoid $\alpha$-powers for $\alpha > \beta$. It is known that these shift spaces have positive topological entropy for many choices of $(d,\beta)$. The main result of this paper is the following.

\begin{theorem}\label{thm:main}
For every $d\geq 2$ and $\beta > 12$, the shift spaces $\Xdb$ and $\Xdbp$ both have a unique measure of maximal entropy (MME). 
\end{theorem}

In the course of the proof, we establish the following counting bounds.

\begin{theorem}\label{thm:card}
Given a shift space from Theorem \ref{thm:main}, let $\LLL_n$ be the set of length-$n$ words in its language (the set of extendable $\beta$-free $d$-ary words of length $n$), and $h = \lim_{n\to\infty} \frac 1n \log\#\LLL_n > 0$ its topological entropy.
Then for every $n\in \NN$, we have
\begin{equation}\label{eqn:card}
e^{nh} \leq \#\LLL_n \leq \begin{cases}
(2.892) e^{(n+2) h} &\text{if } d=2, \\
\big(1 + \frac{d}{d^3 - 2d - 1} \big)^2 e^{(n+1) h} &\text{if } d\geq 3.
\end{cases}
\end{equation}
\end{theorem}

\begin{remark}\label{rmk:lit}
The set of parameter values $(d,\beta)$ in these results is not optimal, and neither is the upper bound in \eqref{eqn:card}. 
When $\beta>3$ and $d$ is sufficiently large, a result of Pavlov \cite[Theorem 7.2]{rP22} can be used to deduce uniqueness of the MME together with counting bounds analogous to \eqref{eqn:card}, as described in Theorem \ref{thm:Pavlov} below.
See also Remarks \ref{rmk:lit} and \ref{rmk:Pavlov} for more on the relationship between Pavlov's arguments and the approach used here.
Using more combinatorial techniques, Bui and Rosenfeld have also established a stronger version of \eqref{eqn:card} \cite{BR25}.
See \S\ref{sec:lit} for more discussion of these and other related results in the literature.

Beyond the question of which parameter values lead to uniqueness of the MME, we would like to understand the ergodic properties of that MME.
 Pavlov uses a result of Ledrappier to show that the unique MME in \cite{rP22} has the K property, and it seems reasonable to expect this to hold in Theorem \ref{thm:main} as well. On the other hand,
moving beyond the K property, the power-free shifts $X_\beta^d$ lack the periodic orbits that are characteristic of local product structure, a key ingredient in the proof of the Bernoulli property in hyperbolic smooth dynamics, and thus the following question seems particularly intriguing:
\begin{quote}
\emph{Is the unique MME in Theorem \ref{thm:main} measure-theoretically isomorphic to a Bernoulli process?}
\end{quote}
\end{remark}

\subsection{Main definitions}\label{sec:def-intro}

To define $\Xdb$ and $\Xdbp$ more carefully, we recall some basic notions from symbolic dynamics and from combinatorics on words. Let $A$ be a finite set (the \emph{alphabet}), $A^\ZZ$ the set of all bi-infinite strings of symbols in $A$, and $\sigma \colon A^\ZZ \to A^\ZZ$ the left shift map defined by $\sigma(x)_n = x_{n+1}$. Equipping $A^\ZZ$ with the product topology, a (two-sided) \emph{shift space} is a closed set $X \subset A^\ZZ$ that is shift-invariant in the sense that $\sigma(X) = X$.

A \emph{(finite) word} is an element $w\in A^n$ for some $n\in \NN \cup \{0\}$; we will write $|w| = n$ to denote the \emph{length} of $w$. The set of all finite words is $A^* := \bigcup_{n\geq 0} A^n$. Given a shift space $X \subset A^\ZZ$, the collection of words
\[
\LLL = \LLL(X) := \{ x_i x_{i+1} \cdots x_j : x\in X,\ i,j\in \ZZ \}
\]
is the \emph{language} of $X$, and we will write $\LLL_n := \LLL \cap A^n$ for the set of words of length $n$ in the language. The language $\LLL$ of a shift space is \emph{factorial} (closed under passing to subwords) and \emph{extendable} (given $w\in \LLL$, there exist $a,b\in A$ such that $awb\in \LLL$).

It is often convenient to describe a shift space and its language via a set of \emph{forbidden words}: given $\mathcal{F} \subset A^+ := \bigcup_{n\geq 1} A^n$, the 
associated shift space is
\[
X = X(\mathcal{F}) = \{ x\in A^\ZZ : x_i x_{i+1} \cdots x_j \notin \mathcal{F} \text{ for all } i,j\in \ZZ \text{ with } i\leq j \}.
\]
That is, $X(\mathcal{F})$ is the set of all infinite sequences over $A$ that have no subword in $\mathcal{F}$.

Given $n\in \NN$, a word $w\in A^n$, and a rational number $\alpha > 1$ such that $\alpha n\in \NN$, the \emph{$\alpha$-power} of $w$ is the word $w^\alpha$ formed by the first $\alpha n$ symbols of $w^\infty = wwww\cdots$. For example, we have
\[
(\mathsf{do})^{2} = \mathsf{dodo},
\qquad
\mathsf{b(an)^{5/2} = banana},
\qquad
\mathsf{(alf)^{7/3} = alfalfa}.
\]

\begin{definition}
For a natural number $d\geq 2$ and a real number $\beta>1$, the \emph{$d$-ary $\beta$-free shift space} $\Xdb$ is the shift space defined by the set of forbidden words
\begin{equation}\label{eqn:Fb}
\FFF_\beta := \{ v^\alpha : v\in A^+,\ \alpha \in [\beta,\infty) \cap \mathbb{Q} \}.
\end{equation}
In other words, $\Xdb = X(\FFF_\beta)$ consists of all $x\in A^\ZZ$ that remain after we remove all strings containing an $\alpha$-power for any $\alpha\geq \beta$.
Replacing $[\beta,\infty)$ with $(\beta,\infty)$ in \eqref{eqn:Fb} gives a set $\FFF_{\beta^+}$ of forbidden words that determines a subshift $\Xdbp$.
\end{definition}

If $\beta \in \RR \setminus \mathbb{Q}$, then $\Xdbp = \Xdb$. From now on, $\beta$ will denote either a real number or a ``rational number with a $+$''. The set of possible $\beta$ carries the obvious order, so we could write Theorem \ref{thm:main} as: ``For every $d\geq 2$ and $\beta \geq 12^+$, the shift space $\Xdb$ has a unique MME''. (In fact, for $d\geq 3$, we prove the result for $\beta = 12$ as well; see Theorem \ref{thm:get-spec}.)

\begin{remark}\label{rmk:extendable}
Given $\FFF \subset A^+$, and writing $\tilde{\LLL}(\mathcal{F})$ for the set of all $w\in A^*$ that have no subword in $\mathcal{F}$, we have $\LLL(X(\FFF)) \subset \tilde{\LLL}(\FFF)$, but the inclusion can be proper: the language $\tilde{\LLL}(\FFF)$ is factorial, but need not be extendable, and $\LLL(X(\FFF))$ is its extendable part, characterized as the set of all $w\in A^*$ that can be extended to a bi-infinite sequence $x\in A^\ZZ$ with no subword in $\mathcal{F}$. The ratio $\#\tilde{\LLL}_n / \#\LLL_n$ is subexponential in $n$ \cite{aS08}, but can be unbounded.
In particular, we stress that when $X = \Xdb$ is the $d$-ary $\beta$-free shift space, its language $\LLL$ consists of the \emph{extendable} $\beta$-free words, and it is to this set that the bounds in Theorem \ref{thm:card} apply.
\end{remark}

\subsection{Overview of paper}\label{sec:overview}

See \S\ref{sec:ent} for a discussion of entropy in the context of symbolic dynamics, and for two results that together imply Theorems \ref{thm:main} and \ref{thm:card}:
\begin{itemize}
\item Theorem \ref{thm:get-spec} formulates a weak analogue of Bowen's specification property that is satisfied by $\Xdb$, and which quickly leads to the counting estimates that establish Theorem \ref{thm:card};
\item Theorem \ref{thm:general} uses this specification-type property to obtain uniqueness of the MME, proving Theorem \ref{thm:main}.
\end{itemize}
Theorems \ref{thm:get-spec} and \ref{thm:general} are proved in \S\ref{sec:power-free} and \S\ref{sec:pf-gen}, respectively.

See \S\ref{sec:lit} for a review of some prior work on power-free languages, and a description of the set of $(\beta,d)$ for which $\Xdb$ is known or conjectured to have positive entropy; as indicated in Remark \ref{rmk:lit},
this set is rather larger than the set of parameters to which Theorems \ref{thm:main} and \ref{thm:card} apply.
See \S\ref{sec:open} for a discussion of this issue, and some other open problems.
With this in mind, both of these main results should be regarded in some sense as ``proofs of concept'' demonstrating the utility of specification-type techniques in this setting, rather than as attempts to wring out the strongest possible results from these ideas: in order to present the main ideas as efficiently and readably as possible, no serious attempt has been made here to push these techniques to the fullest extent.

\subsection{Acknowledgments}

The author is grateful to Ronnie Pavlov for helpful discussions concerning the application of the result in \cite{rP22} to power-free shifts, and in particular for providing the proof of Theorem \ref{thm:Pavlov} below.

\section{Background, definitions, and related questions}\label{sec:def}

\subsection{Entropy}\label{sec:ent}

We describe the main concepts we will need from thermodynamic formalism, as they appear in symbolic dynamics. See \cite{pW82} for an introduction to entropy and the variational principle, and \cite{CT21} for a survey of the development of specification-type properties in the study of thermodynamic formalism.

Given a finite alphabet $A$ and a shift space $X \subset A^\ZZ$, recall that $\LLL \subset A^*$ denotes the language of $X$, with $\LLL_n = \LLL \cap A^n$ denoting the set of words of length $n$ in the language. Equivalently, $\LLL_n$ is the set of words $w\in A^n$ for which the \emph{cylinder}
\[
[w] := \{ x\in X : x_1 \cdots x_n = w \}
\]
is nonempty. The \emph{topological entropy} of $X$ is
\begin{equation}\label{eqn:h}
h = h(X) := \lim_{n\to\infty} \frac 1n \log \#\mathcal{L}_n,
\end{equation}
where the limit exists by Fekete's lemma since $\#\LLL_{n+k} \leq (\#\LLL_n)(\#\LLL_k)$.

An \emph{invariant measure} on $X$ is a Borel probability measure $\mu$ such that 
for every Borel set $E\subset X$, we have $\mu(\sigma^{-1}E) = \mu(E)$.
The \emph{measure-theoretic entropy} of $\mu$ is
\begin{equation}\label{eqn:hmu}
h_\mu := \lim_{n\to\infty} \frac 1n H_n(\mu),
\quad\text{where } H_n(\mu) := \sum_{w\in \LLL_n} -\mu([w]) \log \mu([w]);
\end{equation}
existence of the limit is again a consequence of Fekete's lemma and the subadditivity property $H_{n+k}(\mu) \leq H_n(\mu) + H_k(\mu)$.
Writing $\MMM_\sigma(X)$ for the space of invariant measures, the \emph{variational principle} says that
\begin{equation}\label{eqn:vp}
h(X) = \sup \{ h_\mu : \mu \in \MMM_\sigma(X) \}.
\end{equation}
A measure achieving the supremum is called a \emph{measure of maximal entropy} (MME). Every shift space on a finite alphabet has an MME because the function $\mu \mapsto h_\mu$ is upper semi-continuous in the weak* topology, in which $\MMM_\sigma(X)$ is compact.

It is an important question in ergodic theory and thermodynamic formalism to understand for which classes of shift spaces the MME is unique, and to describe its properties; see \S\ref{sec:open} for more. There are examples of shifts with multiple MMEs, but uniqueness is known to hold for every topologically transitive subshift of finite type (including the full shift), every transitive sofic shift, and many other examples.

We will focus our attention on a variant of Rufus Bowen's \emph{specification} property. For a shift space, this property reduces to the following:\footnote{Here we use the notation $w^{(j)}$ for an element of an indexed set of words, reserving the notation $w^j$ for the power of $w$ obtained by concatenating $j$ copies, and the notation $w_j$ for the $j$th symbol of a word $w$. We point out that there are many versions of the specification property in the literature; see \cite{kY09,KLO16} for overviews. The property in \eqref{eqn:Bow-spec} is a uniform specification property that imposes no \emph{a priori} requirement on existence of a periodic shadowing orbit.}
\begin{multline}\label{eqn:Bow-spec}
\text{there exists $\tau \in \NN$ such that for every $k\in \NN$ and $w^{(1)},\dots, w^{(k)} \in \LLL$,} \\
\text{there exist $u^{(1)},\dots, u^{(k-1)} \in \LLL_\tau$ such that $w^{(1)} u^{(1)} w^{(2)} \cdots u^{(k-1)} w^{(k)} \in \LLL$.}
\end{multline}
It can quickly be seen that this is equivalent to: \begin{multline}\label{eqn:Bow-spec-2}
\text{there exists $\tau\in \NN$ such that for every $v,w\in \LLL$,}\\
\text{there exists $u\in \LLL_\tau$ such that $vuw\in \LLL$.}
\end{multline}
Every topologically mixing sofic shift (in particular, every topologically mixing subshift of finite type) has specification. Bowen proved \cite{Bow75} that every shift space with specification has a unique MME.

Specification produces periodic orbits: writing $P_n := \{ x\in X : \sigma^n(x) = x\}$, one can prove that $\frac 1n \log \# P_n \to h$ whenever \eqref{eqn:Bow-spec-2} is satisfied.\footnote{To see this, use \cite{aB88} to get a synchronizing word $v$, and then for each $w\in \LLL_n$, use \eqref{eqn:Bow-spec} to get $u,u' \in \LLL_\tau$ such that $vuwu'v \in \LLL$, which implies $x := (vuwu')^\infty \in P_{n+2\tau+|v|}$.}
Since the power-free shifts $\Xdb$ have no periodic orbits, they cannot satisfy Bowen's specification property. To formulate a weaker property that they do satisfy, recall
that a \emph{prefix} of $w\in A^n$ is a word of the form $w_{[1,k]} := w_1 w_2 \cdots w_k$ for some $k \in \{1,\dots, n\}$, and a \emph{suffix} of $w$ is a word of the form $w_{[k,n]}$. The following will be proved in \S\ref{sec:power-free}.

\begin{theorem}\label{thm:get-spec}
Fix $(d,\beta)$ such that $d\geq 3$ and $\beta \geq 12$, or $d=2$ and $\beta \geq 12^+$.
Let $X = \Xdb$ be the corresponding $d$-ary $\beta$-free shift, and $\LLL \subset \{1,\dots,d\}^*$ its language. Define $\Cp,\GGG,\Cs \subset \LLL$ as follows:\footnote{Here $\mathcal{A}^*$ uses the Kleene star to denote arbitrary finite (possibly empty) concatenations of words from $\mathcal{A}$, just as $A^*$ denoted arbitrary concatenations of symbols from the alphabet.}
\begin{equation}\label{eqn:CGC}
\begin{gathered}
\GGG := \{ w \in \LLL : \text{no prefix or suffix of $w$ is a $4$-power} \}, \\
\Cp = \Cs := \AAA^* \cap \LLL,
\quad\text{where }
\AAA := \{ v^4 : v\in \LLL\}.
\end{gathered}
\end{equation}That is, $\Cp = \Cs$ is the set of words in $\LLL$ that can be written as a concatenation of $4$-powers.
Then the following are true.
\begin{enumerate}[leftmargin=*, label=\upshape{(\Roman{*})}, itemsep=1ex]
\item\label{decomp}
(Decomposition):
for every $w\in \LLL$, there exist $v\in \GGG$ and $u^{p,s} \in \mathcal{C}^{p,s}$ such that $w=u^p v u^s$.
\item\label{h-gap} 
(Entropy gap):
$h(\Cp \cup \Cs) := \varlimsup_{n\to\infty} \frac 1n \log \#(\CCC^p \cup \CCC^s)_n < h(X)$.
\item\label{spec-diff} 
(Variable-length specification):
There exists $T\in \NN$ such that for all $u,v,w,x\in \GGG$, there are $p,q,r \in \LLL_{T}$ such that $upvqwrx\in \LLL$.
\item\label{spec-same} 
(Same-length specification): There exists $\tau\in \NN$ such that for every $v,w\in \GGG$ with $|v| = |w|$, there exists $u\in \LLL_\tau$ such that $vuw\in \GGG$.
\end{enumerate}
Moreover, $\#\LLL_n$ satisfies the bounds in \eqref{eqn:card}.
\end{theorem}

Conditions \ref{decomp} and \ref{h-gap} require a decomposition of the language in which the prefix and suffix collections $\Cp$ and $\Cs$ do not carry full entropy. Conditions \ref{spec-diff} and \ref{spec-same} are weakenings of the specification properties \eqref{eqn:Bow-spec} and \eqref{eqn:Bow-spec-2}. In particular, shifts with specification satisfy \ref{decomp}--\ref{spec-same} by taking $\Cp=\Cs=\emptyset$ and $\GGG = \LLL$.

We will prove that every shift space satisfying \ref{decomp}--\ref{spec-same} has a unique MME.  Our argument will use a slightly weaker version of condition \ref{spec-same}: the following is proved in \S\ref{sec:spec-growth} by iterating \ref{spec-same}.

\begin{lemma}\label{lem:general-spec}
Let $\LLL$ be the language of a shift space $X$, and suppose that $\GGG \subset \LLL$ satisfies \ref{spec-same}. Then $\GGG$ satisfies the following.
\begin{enumerate}[leftmargin=*, label=\upshape{(\Roman{*}$'$)}]\setcounter{enumi}{3}
\item\label{spec-same-2}
There exists $\tau\in \NN$ such that for every $k,n \in \NN$ and $w^{(1)},\dots, w^{(k)} \in \GGG_n$,
there exist $u^{(1)},\dots, u^{(k-1)} \in \LLL_{\tau}$ such that $w^{(1)} u^{(1)} w^{(2)} u^{(2)} \cdots u^{(k-1)} w^{(k)} \in \LLL$.
\end{enumerate}
\end{lemma}

Our general uniqueness result will also use
the following consequence of \ref{h-gap} instead of \ref{h-gap} itself.
\begin{enumerate}[leftmargin=*, label=\upshape{(\Roman{*}$'$)}]\setcounter{enumi}{1}
\item\label{PR}
The quantity $Q := (\sum_{i=0}^\infty \#\Cp_i e^{-ih})(\sum_{k=0}^\infty \#\Cs_k e^{-kh})$ is finite.
\end{enumerate}
The following theorem is proved in \S\ref{sec:pf-gen}. Together with Theorem \ref{thm:get-spec} and Lemma \ref{lem:general-spec}, it completes the proof of Theorem \ref{thm:main}.

\begin{theorem}\label{thm:general}
Let $X$ be a shift space on a finite alphabet with language $\LLL$. Suppose that there exist $\CCC^p,\GGG,\CCC^s \subset \LLL$ satisfying \ref{decomp}, \ref{PR}, \ref{spec-diff}, and \ref{spec-same-2}. Then $X$ has a unique measure of maximal entropy. Moreover, we have $e^{nh} \leq \#\LLL_n \leq Q e^{(n+\tau)h}$ for all $n$, where $Q$ and $\tau$ are as in \ref{PR} and \ref{spec-same-2}.
\end{theorem}

\begin{remark}[Other non-uniform specification results]
Conditions analogous to those appearing in Theorem \ref{thm:general} were introduced by the author and Dan Thompson in \cite{CT12}. In that paper, \ref{spec-diff} and \ref{spec-same-2} were replaced by the stronger (but simpler to formulate) condition that ``the specification property is satisfied for words in $\GGG$'', in the sense that
\begin{multline}\label{eqn:spec}
\text{for every $k\in \NN$ and $w^{(1)},\dots, w^{(k)} \in \GGG$,} \\
\text{there exist $u^{(1)},\dots, u^{(k-1)} \in \LLL_\tau$ such that $w^{(1)} u^{(1)} w^{(2)} \cdots u^{(k-1)} w^{(k)} \in \LLL$.}
\end{multline}
This condition, together with \ref{decomp} and \ref{PR}, implies uniqueness of the MME. This was proved in \cite{CT12,CT13} modulo an extra assumption on ``extending to words in $\GGG$''. This extra condition was ultimately proved to be unnecessary by Pacifico, Fan Yang, and Jiagang Yang \cite{PYY22}. Theorem \ref{thm:general} is another generalization of the result from \cite{CT12}: observe that \ref{spec-diff} and \ref{spec-same-2} are weakened versions of \eqref{eqn:spec}, obtained by restricting to the case $k=4$ in \ref{spec-diff}, and to words of the same length in \ref{spec-same-2}.

It is also worth comparing Theorem \ref{thm:general} to \cite[Theorem 3.2]{rP22}, where Pavlov uses ideas from his earlier work in \cite{rP19} to prove uniqueness of the MME under the following conditions:
\begin{itemize}
\item there exists $C \geq 1$ such that $\#\LLL_n \leq C e^{nh}$ for all $n$;
\item there exists $\GGG \subset \LLL$ satisfying a weaker version of \ref{spec-diff}, which only requires the ability to glue two words (instead of four);
\item the collection $\GGG$ can also be chosen to have the property that for every ergodic MME $\mu$, there exists a syndetic set $S \subset \NN$ such that $\inf_{n\in S} \mu(\GGG_n) > 0$.
\end{itemize} 
Pavlov's argument for uniqueness is rather different from the one given here; his proof uses the maximal ergodic theorem, and does not rely on establishing any sort of Gibbs property. Indeed, he proves a Gibbs bound \emph{after} proving uniqueness, whereas we proceed in the opposite order. It would be interesting to understand whether his techniques could be used to sharpen the results here.
\end{remark}

\begin{remark}
Echoing the comments preceding Theorem \ref{thm:get-spec}, 
we emphasize that the absence of periodic points in the power-free shifts $\Xdb$ stands in stark contrast to prior applications of specification-type properties, such as \cite{CT12,Cli18,BCFT18}, where periodic orbits are abundant and grow as quickly as the topological entropy.
\end{remark}

\subsection{Growth in power-free languages}\label{sec:lit}

Power-free languages have been well-studied. For a full introduction to these languages, and other topics in combinatorics on words, we refer to \cite{mL83,AS03}.
We will focus on the growth properties of power-free languages; overviews of the literature on this topic can be found in the surveys by Berstel \cite{jB05} and Shur \cite{aS11,aS12}. 

\subsubsection{Repetition threshold and Dejean's conjecture}

A natural first question is to ask when $X_\beta^d$ is non-empty. One quickly sees that it is empty when $d=2$ and $\beta\leq 2$: the only nonempty words over $A = \{\mathsf{0,1}\}$ that do not contain a square are
\[
\mathsf{0},\quad
\mathsf{01},\quad
\mathsf{010},\quad\text{and}\quad
\mathsf{1},\quad\mathsf{10},\quad\mathsf{101}.
\]
On the other hand, the Thue--Morse sequence\footnote{This sequence is  $\mathsf{0 1 10 1001 10010110\dots}$, where the
first $2^n$ symbols determine the next $2^n$ symbols via the morphism $\mathsf{0\leftrightarrow 1}$.}
does not contain any subword that is an $\alpha$-power for any $\alpha > 2$ (it is \emph{overlap-free}), so $X_{2^+}^2 \neq \emptyset$. Thus $\beta = 2$ is the \emph{repetition threshold} for the alphabet size $d=2$: writing
\[
\mathsf{RT}(d) := \inf \{ \beta > 1 : X_\beta^d \neq \emptyset \},
\]
we have $\mathsf{RT}(2) = 2$.
In 1972, Dejean proved that $\mathsf{RT}(3) = \frac 74$, and conjectured that
\[
\mathsf{RT}(4) = \frac 75,
\qquad
\mathsf{RT}(d) = \frac{d}{d-1}
\text{ for all } d\geq 5.
\]
This was eventually proved for all $d$ through a combination of work by Pansiot \cite{jP84}, Moulin-Ollagnier \cite{MO92}, Mohammad-Noori and Currie \cite{MC07}, Carpi \cite{aC07}, Currie and Rampersad \cite{CR11}, and Rao \cite{mR11}.

\subsubsection{Entropy and the exponential conjecture}

With the repetition thresholds known, it is natural to ask for which values of $(\beta,d)$ with $\beta > \mathsf{RT}(d)$ the topological entropy of $X_\beta^d$ is positive.\footnote{The literature in combinatorics on words tends to use the term \emph{exponential complexity} rather than \emph{positive topological entropy}, but the meaning is the same.}
 That is, we want to determine the set
\[
\mathcal{P} := \{ (\beta,d) \in (1,\infty) \times \{2,3,\dots\} : h(X_\beta^d) > 0 \}.
\]
Early results in this direction were given by Brandenburg, who proved that $X_3^2$ and $X_2^3$ have positive topological entropy \cite{fB83}. 

In the other direction,
Restivo and Salemi \cite{RS85} proved that $h(X_{2^+}^2) = 0$ by characterizing elements of this shift space via the Thue--Morse morphism; this was extended to $X_\beta^2$ for $2 < \beta \leq \frac 73$ by Karhum\"aki and Shallit \cite{KS04}.\footnote{Both \cite{RS85} and \cite{KS04} actually obtained more explicit polynomial bounds on 
$\#\LLL_n(X_\beta^2)$, but the subexponential nature of the growth is the main phenomenon we are interested in here.} 

This turns out to be sharp: Karhum\"aki and Shallit also proved that $h(X_\beta^2) > 0$ for all $\beta > \frac 73$. 
When $d\geq 3$, no analogue of the ``polynomial plateau'' $(2,\frac 73] \times \{2\}$ has been observed. Ochem proved that for $d=3,4$, every nonempty $X_\beta^d$ has positive entropy, and conjectured that this is true for all $d\geq 3$ \cite[Conjecture 4.4]{pO06}. An equivalent formulation of this \emph{exponential conjecture} is that, writing $\beta = \mathsf{RT}(d)$, the shift space $\Xdbp$ has positive topological entropy.
This has now been proved for all $d\geq 3$ except for the even numbers from $12$ to $26$ \cite{KR11,TS12,CMR20}.

\begin{figure}[htbp]
\begin{tikzpicture}[yscale=0.6]
\fill[orange!50!white] (2,0)--(7/3,0)--(7/3,1)--(2,1)--cycle;
\fill[blue!25!white] (7/3,0)--(7/3,1)--(7/4,1)--(7/4,2)%
--(7/5,2)--(7/5,3)--(5/4,3)--(5/4,4)%
--(8,4)--(8,0)--cycle;
\fill[blue!25!white, path fading=north]
(8,4)--(6/5,4)--(6/5,5)--(8,5)--cycle;
\node at (0,.5)[left] {$2$};
\node at (0,1.5)[left] {$3$};
\node at (0,2.5)[left] {$4$};
\node at (0,3.5)[left] {$5$};
\draw[->] (0,0)--(11.5,0) node[right]{$\beta$};
\draw[->] (0,0)--(0,5) node[left]{$d$};
\tikzfading[name=fade out, 
    inner color=transparent!0,
    outer color=transparent!100] 
\fill[green!25!white,path fading=fade out] (10,4) circle[radius=1];
\fill[white] (9,3) rectangle (10,5);
\fill[white] (10,3) rectangle (11,4);
\fill[green!25!white, path fading=north]
(8,4) rectangle (10,5);
\fill[green!25!white, path fading=east] 
(10,0) rectangle (11,4);
\fill[green!25!white]
(8,0) rectangle (10,4);
\draw[dotted, ultra thick] (7/3,0)--(7/3,1)--(7/4,1)--(7/4,2)%
--(7/5,2)--(7/5,3)--(5/4,3)--(5/4,4)--(6/5,4)--(6/5,5);
\draw[ultra thick] (8,1)--(8,5);
\draw[ultra thick, dotted] (8,0)--(8,1);
\draw[gray] (0,1)--(11.5,1);
\draw[gray] (0,2)--(11.5,2);
\draw[gray] (0,3)--(11.5,3);
\draw[gray] (0,4)--(11.5,4);
\draw[black!80!white] (7/4,1)--(7/4,-.3) node[below left]{$\mathsf{RT}(d)$};
\draw[black!80!white] (7/5,2)--(7/5,-.3);
\draw[black!80!white] (5/4,3)--(5/4,-.3);
\draw[black!80!white] (6/5,4)--(6/5,-.3);
\draw (2,1)--(2,-.3) node[below]{$2$};
\draw (7/3,0)--(7/3,-.3) node[below right]{$\frac 73$};
\draw (8,0)--(8,-.3) node[below]{$12$};
\draw[red,->] (0,6) node[left]{$\emptyset$} -- (0.6,5)--(0.6,2.5);
\draw[red,->] (4,5.5) node[above]{$h>0$, uniqueness of MME unknown} -- (4,2.5);
\draw[red,->] (10,5.5) node[above]{$h>0$, unique MME} -- (9,2.5);
\draw[red,->] (5,-.5) node[below]{polynomial plateau}
-- (3,.5) -- (2.2,.5);
\end{tikzpicture}
\caption{A map of parameter space for $\Xdb$.}
\label{fig:beta-d}
\end{figure}

\subsubsection{Uniqueness of the MME}

The discussion so far is summarized in Figure \ref{fig:beta-d}, which illustrates how Theorem \ref{thm:main} can be viewed as a step towards answering the following question.

\begin{question}\label{Q:every}
Does every power-free shift space $X_\beta^d$ with positive topological entropy have a unique MME?
\end{question}

As mentioned in Remark \ref{rmk:lit} and \S\ref{sec:overview},
the arguments given in \S\ref{sec:power-free} have been written down with brevity and readability in mind, rather than with the goal of finding the largest set of $(\beta,d)$ to which these specification-type techniques can be applied.

The conclusions of Theorems \ref{thm:main} and \ref{thm:card} hold for many smaller values of $\beta$, at least when the alphabet is sufficiently large. For example, Pavlov has shown  that if $X$ is a subshift on a $d$-letter alphabet $A$ defined by a set of forbidden words that contains at most $b_n$ words of length $n$, and if
\begin{equation}\label{eqn:sum-small}
\sum_{n=1}^\infty n^2 b_n (3d^{-1})^{n/3} < \frac 1{36},
\end{equation}
then $X$ has a unique MME, which has the K-property \cite[Theorem 7.2]{rP22}. The brief argument establishing the following consequence of this result was provided to the author by Ronnie Pavlov.

\begin{theorem}[Pavlov]\label{thm:Pavlov}
For every $\beta>3$, there exists $d_0\in \NN$ such that for all $d\geq d_0$, both $\Xdb$ and $\Xdbp$ have a unique MME, and this measure has the K-property. Moreover, $\#\LLL_n \leq 4e^{nh}$ for all sufficiently large $n$.
\end{theorem}
\begin{proof}
First observe that given any $N\in \NN$, the number of non-empty words with length $\leq N$ is at most
$\sum_{\ell=1}^N d^\ell = d \big(\frac{d^N - 1}{d-1}\big) \leq 2d^N$. Now given $n\in \NN$, a forbidden word of length $n$ must have the form $w^{p/q}$, where $\frac pq \leq \beta$, and thus $|w| \leq n/\beta$. 
Thus we can take $b_n = 2d^{n/\beta}$, and the sum in \eqref{eqn:sum-small} becomes
\[
\sum_{n=1}^\infty n^2 2d^{n/\beta} (3d^{-1})^{n/3}
= 2\sum_{n=1}^\infty n^2 \lambda_d^n
= \frac{2\lambda_d(1+\lambda_d)}{(1-\lambda_d)^3},
\quad\text{where }
\lambda_d := 3^{\frac 13} d^{\frac 1\beta - \frac 13},
\]
provided $\lambda_d \in (0,1)$. Since $\beta>3$, we have $\lambda_d \to 0$ as $d\to\infty$, so \eqref{eqn:sum-small} is satisfied for all sufficiently large $d$. Together with \cite[Theorems 4.3 and 7.2]{rP22}, this proves the result.
\end{proof}

\begin{remark}\label{rmk:Pavlov}
There are many similarities between the arguments in \cite{rP22} and the arguments here. 
In both cases, the overall approach is to formulate abstract conditions that are satisfied by the class of subshifts to be studied: subshifts with slow forbidden word growth in \cite{rP22}, and power-free shifts here. The similarity between the abstract conditions was discussed in Remark \ref{rmk:lit}. The arguments that Pavlov uses to verify his conditions for subshifts with slow forbidden word growth are based on ideas of Miller \cite{jM12} and his own prior work on coded systems \cite{rP20}. In the specific case of power-free shifts, there are strong parallels to the approach here. The following observation is indicative: if we consider the case $\beta=12$, so that $\FFF = \{ v^{12} : v\in \LLL\}$, then the collection of words $\AAA$ defined in \eqref{eqn:CGC} is precisely the set of heavy subwords in the sense of \cite[Definition 5.1]{rP22}, and the collection $\GGG$ in \eqref{eqn:CGC} agrees with the set of words $G$ from \cite[Definition 5.2]{rP22}.
\end{remark}

Regarding the range of parameters that one could hope to treat using Theorem \ref{thm:general}, we make the following observations.
For $d\geq 3$ and $\beta\geq 16$, taking $\mathcal{A} = \{v^4 : v\in \LLL\}$ will let us verify \ref{spec-diff} and \ref{spec-same} with transition times $T=0$ and $\tau=1$, while going all the way to $d=2$ and $\beta = 12^+$ requires $T=\tau=2$. It seems likely that the same choice of $\AAA$ could be made to work for at least some smaller values of $d$ and $\beta$ at the cost of increasing the transition times and the lengths of the proofs, and one might also try to go further by replacing $v^4$ with a smaller power in the definition of $\mathcal{A}$. 

For purposes of the uniqueness argument, the exact values of the constants in the counting bounds \eqref{eqn:card} are not important; the crucial thing is that such constants exist. One might hope to improve these bounds using the approach here by making more careful estimates of $\#(\AAA^*)_n$ and $\#\GGG_n$. It deserves to be mentions that improved bounds can also be obtained by other methods: this was done by Bui and Rosenfeld \cite{BR25}, using combinatorial techniques introduced in \cite{mR25}. The crucial ingredient is an explicit estimate of a constant $C \in (0,1)$ such that $\#\LLL_{n+m} \geq C \#\LLL_n \#\LLL_m$ for all $n,m\in \NN$ \cite[Theorems 10 and 11]{BR25}. These results apply to a broader range of parameters than Theorems \ref{thm:card} or \ref{thm:Pavlov}, and it would be interesting to see if they could be used to study uniqueness of the MME for these parameters as well.

\subsection{Related questions}\label{sec:open}

The surveys \cite{jB05,aS11,aS12} describe related results and questions concerning growth in power-free languages. These include techniques for numerically approximating $h(X_\beta^d)$; see \cite[\S\S3--5]{aS12} for a discussion of these, and of the continuity properties of $\beta \mapsto h(X_\beta^d)$. It would be interesting to study the continuity properties not just of the topological entropy, but also of the dependence of the MME on the parameter $\beta$.

Another direction that has been studied is to consider \emph{repetition-free} languages and shift spaces defined by variants of the power-free condition. For example, one could allow powers of short words but not of longer words \cite{EJS74,fD76,FS95,RSW05}, or could impose a restriction that we avoid patterns other than pure powers \cite{fD79,CKX12}. It seems likely that the methods developed in this paper could be useful in at least some such examples.

We conclude this section by mentioning some further questions drawn from the theory of uniformly hyperbolic dynamical systems, where thermodynamic formalism was developed and where the specification-type properties used in Theorem \ref{thm:general} assume their strongest form. In this setting, uniqueness of the MME is only the beginning of the story. For example, given a shift space $X$ with the specification property \eqref{eqn:Bow-spec}, the following are true.
\begin{itemize}
\item The unique MME $\mu$ is fully supported: it gives positive weight to every open set \cite{Bow75}. Moreover, the system $(X,\sigma,\mu)$ is measure-theoretically isomorphic to a Bernoulli shift \cite{rB-Bern,yD13,Cli18}: in particular, it has the K property\footnote{The K (Kolmogorov) property
follows from Bernoullicity and implies mixing; it is equivalent to \emph{completely positive entropy}, meaning that $(X,\sigma,\mu)$ has positive entropy over any partition.} 
\cite{fL77} and is mixing.
\item For every H\"older continuous function $\ph \colon X\to \RR$, there is a unique \emph{equilibrium state} $\mu$ that maximizes the quantity $h_\mu + \int \ph \,d\mu$. This equilibrium state is fully supported and Bernoulli.
\item The space of invariant measures $\MMM_\sigma(X)$ is very rich: its set of extreme points (ergodic measures) is entropy-dense  \cite{EKW94,PS05}, which implies that $\MMM_\sigma(X)$ is linearly isomorphic to the Poulsen simplex; moreover, it contains measures isomorphic to every ergodic aperiodic transformation with entropy less than $h$ \cite{wK75,QS16}.
\end{itemize}

\begin{question}\label{Q:MME}
Fix $d,\beta$ as in Theorem \ref{thm:get-spec}, and let $\mu$ be the unique MME on $X_\beta^d$.
\begin{enumerate}[leftmargin=*,label=\upshape{(\alph{*})}]
\item Is $\mu$ fully supported?\label{Q:supp}
\item Does $\mu$ have the K property? Is it Bernoulli?
\end{enumerate}
\end{question}

The question of full support is closely related to what Shallit and Shur call the \emph{Restivo--Salemi} property \cite[\S4]{SS19}, after a question posed in
\cite[Problem 4]{RS85}: this is the condition that given any $u,v\in \LLL(X_\beta^d)$, there exists a word $w$ such that $uwv \in \LLL(X_\beta^d)$. In the terminology of dynamical systems, this is the property that $(X_\beta^d,\sigma)$ is topologically transitive, which must be true whenever there exists a fully supported ergodic measure; in particular, a positive answer to Question \ref{Q:MME}\ref{Q:supp} would imply the Restivo--Salemi property. Numerical evidence suggests that the Restivo--Salemi property should hold for every $X_\beta^d$; see \cite[Conjecture 1]{aS09} and \cite[Conjecture 42]{SS19}. However, the author is not aware of a proof of this property for any $X_\beta^d$ with positive entropy.

For the parameters $d,\beta$ in Theorem \ref{thm:Pavlov}, Pavlov's result from \cite{rP22} establishes the K property for the unique MME, and one may hope to obtain a similar result for the parameters in Theorem \ref{thm:main}.
It is less clear to the author whether to expect the Bernoulli property to hold.
As mentioned in Remark \ref{rmk:lit}, Bernoullicity of the MME in hyperbolic dynamics is generally proved using local product structure.
This same local product structure is also associated to the abundance of periodic orbits in hyperbolic dynamics.
Since these are absent in $\Xdb$, and there is no obvious product structure to use, it is not clear how to approach the question of Bernoullicity. 

Moving beyond the MME, the results mentioned above for uniformly hyperbolic systems suggest the following.

\begin{question}\leavevmode
\begin{enumerate}[leftmargin=*, label=\upshape{(\alph{*})}]
\item
Under what conditions does a H\"older continuous potential function $\ph \colon X_\beta^d \to \RR$ has a unique equilibrium state? 
\item
Is the set of ergodic measures entropy-dense in $\MMM_\sigma(X_\beta^d)$? That is, given any invariant $\mu$, is there a sequence of ergodic measures $\mu_n$ such that $\mu_n\to \mu$ in the weak* topology and $h_{\mu_n} \to h_\mu$? If not, is there at least a weak*-density result, so that $\MMM_\sigma(X_\beta^d)$ is linearly isomorphic to the Poulsen simplex?
\item 
Does the space of invariant measures $\MMM_\sigma(X_\beta^d)$ contain measures isomorphic to every ergodic aperiodic transformation with entropy less than $h$?
If not, does it at least contain measures achieving every entropy between $0$ and $h$?
\end{enumerate}
\end{question}

\section{Specification-type properties for power-free shifts}\label{sec:power-free}

This section is devoted to the proof of Theorem \ref{thm:get-spec}. Fixing $d$ and $\beta$, we write $A = \{1,\dots, d\}$, and consider the following collections of words as in \eqref{eqn:CGC}.
\begin{itemize}
\item Let $\GGG$ be the set of all words $w\in \LLL = \LLL(\Xdb)$ such that $w$ does not begin or end with a word of the form $v^4$ for any $v\in \LLL$. 
\item Let $\AAA = \{ v^4 : v\in \LLL \}$, and let $\Cp = \Cs = \AAA^* \cap \LLL$. 
\end{itemize}
We quickly see that $\Cp,\GGG,\Cs$ satisfy the decomposition property in \ref{decomp}: given any $w\in \LLL$, let $u^p$ be the longest prefix of $w$ that can be written as a concatenation of words in $\AAA$, and $u^s$ the longest suffix of $w$ that can be written in this way and that has length at most $|w| - |u^p|$. 
Then we have $w = u^p v u^s$, where $v$ does not begin or end with a word in $\AAA$ (by maximality of $u^p$ and $u^s$), and thus $v\in \GGG$. Note that $v$ could be the empty word.

With \ref{decomp} verified, it remains to prove that 
$h(\AAA^*) = \ulim \frac 1n \log \#(\AAA^*)_n < h(X)$ so that \ref{h-gap} is satisfied, and to verify that $\GGG$ satisfies the specification-type properties \ref{spec-diff} and \ref{spec-same}. After establishing some basic lemmas in \S\ref{sec:periods}, we will prove \ref{spec-diff} and \ref{spec-same} in \S\ref{sec:spec-type}. Then in \S\ref{sec:get-gap}, we will use \ref{spec-same} to 
obtain uniform bounds relating $\#\GGG_n$, $e^{nh}$, and $\#\LLL_n$, which we use to prove \ref{h-gap} and \eqref{eqn:card}.

\subsection{Periods of words}\label{sec:periods}

Given $w=w_1 \cdots w_n \in A^n$, an integer $\ell$ is \emph{a period} of $w$ if $1\leq \ell\leq n$ and $w_i = w_{\ell+i}$ for all $1\leq i\leq n-\ell$; equivalently, if $w = (w_1 \cdots w_\ell)^{n/\ell}$. The set of all periods of a word $w$ will be denoted $\Pi(w)$, and the smallest element of $\Pi(w)$ will be called \emph{the period} of $w$.

\begin{theorem}[Fine--Wilf {\cite[Theorem 1.5.6]{AS03}}]
If $w$ is a word of length $n$ and $\ell,k$ are periods of $w$ such that $\ell + k - \gcd(\ell,k) \leq n$, then $\gcd(\ell,k)$ is a period of $w$.
\end{theorem}

\begin{corollary}\label{cor:FW}
If $w\in A^n$ has a period in $[1,\lceil n/2 \rceil]$, then every such period is a multiple of the least such period.
\end{corollary}

Given $w\in A^*$ and $\ell \in \NN$, let $\pre_\ell(w)$ be the number of times the word $w_{[1,\ell]}$ is repeated as a prefix of $w$. That is, $\pre_\ell(w) := \frac k\ell$, where $k$ is maximal with the property that $w_{[1,k]}$ is a prefix of $(w_{[1,\ell]})^\infty$. Similarly, let $\suf_\ell(w)$ be the number of times the length-$\ell$ suffix of $w$ is repeated. 

\begin{remark}\label{rmk:G4}
In the context of $\Xdb$, we see that given $w\in \LLL$, we have $w\in \GGG$ if and only if $\pre_\ell(w) < 4$ and $\suf_\ell(w) < 4$ for all $1\leq \ell \leq |w|$.
\end{remark}

\begin{lemma}\label{lem:block-left}
For all $n\in \NN$ and $v\in A^n$, there exists $a\in A$ such that given any $u\in A^+$ with $u_1 \neq a$, we have $\pre_\ell(vu) = \pre_\ell(v)$ for all $\ell \in \{1,\dots, \lceil n/2 \rceil \}$.
\end{lemma}
\begin{proof}
Observe that if $\ell \in \{1,\dots,\lceil n/2\rceil \}$ is not a period of $v$, then $\pre_\ell(vu) = \pre_\ell(v)$ for every $u\in A^+$. In particular, if $P := \Pi(v) \cap [1,\lceil n/2 \rceil]$ is empty, then any $a\in A$ fulfills the conclusion.

If $P\neq \emptyset$, then let $k = \min P$. Writing $x= (v_{[1,k]})^\infty = (v_1 \cdots v_k)^\infty$ for the corresponding periodic infinite string that has $x_{[1,n]} = v$, let $a = x_{n+1}$ be the next symbol following $v$ in this periodic repetition.
By Corollary \ref{cor:FW}, every period $\ell$ of $v$ in $[1,\lceil n/2 \rceil]$ is a multiple of $k$
and thus has $(v_{[1,\ell]})^\infty = x$.
It follows that $a$ has the desired property: given any $u\in A^+$ such that $u_1 \neq a = x_{n+1}$, we see that $\ell$ is not a period of $vu_1$, so $\pre_\ell(vu) = \pre_\ell(v)$.
\end{proof}

\begin{lemma}\label{lem:partial-spec}
Given $v,w\in A^*$, there exist $a_1,a_2 \in A$ such that 
for every $\tau \in \NN$ and every $u\in A^\tau$ satisfying $u_1 \neq a_1$ and $u_\tau \neq a_2$, we have
\begin{align*}
\pre_\ell(vuw) = \pre_\ell(v) &\text{ for every } \ell \in \{1,\dots, \lceil |v|/2\rceil \}, \\
\suf_\ell(vuw) = \suf_\ell(w) &\text{ for every } \ell \in \{1,\dots, \lceil |w|/2\rceil \}.
\end{align*}
\end{lemma}
\begin{proof}
It suffices to apply Lemma \ref{lem:block-left} twice: once to $v$ to obtain $a_1$, and once to the reversed word $w_{|w|} \cdots w_2 w_1$ to obtain $a_2$.
\end{proof}

\subsection{Specification-type properties}\label{sec:spec-type}

The following lemmas verify conditions \ref{spec-diff} and \ref{spec-same} for all values of $\beta$ and $d$ given in Theorem \ref{thm:get-spec}; in fact, we will see that \ref{spec-same} holds for some smaller values of $\beta$ as well. The lemmas will also show how $T$ and $\tau$ can be chosen smaller when $\beta$ and $d$ are larger. 

\begin{lemma}\label{lem:GGGG}
Given any $d\geq 2$ and $\beta \geq 16$, the collection $\GGG$ satisfies \ref{spec-diff} with $T=0$:
for every $u,v,w,x\in \GGG$, we have $uvwx \in \LLL$.
\end{lemma}
\begin{proof}
Suppose $s = t^\alpha$ is a subword of $uvwx$ for some $t\in A^+$ and some rational $\alpha > 1$. If $s$ is a subword of any of the four words $u, v, w, x$, then we have $\alpha < \beta$, since each of these words lies in $\LLL$.

If $s$ is not a subword of any of these words, then by the definition of $\GGG$, each of these words accounts for less than $4|t|$ of the indices in $s$, so $|s| < 16|t|$. This again implies that $\alpha<16\leq\beta$, proving the lemma.
\end{proof}

We will prove \ref{spec-diff} for the smaller values of $\beta$ after first showing how Lemma \ref{lem:partial-spec} can be used to ``glue words in $\GGG$ together''.

\begin{lemma}\label{lem:GbG}
Given any $d\geq 3$ and $\beta \geq 8$, the collection $\GGG$ satisfies \ref{spec-same} with $\tau = 1$.
Indeed, given $n\in \NN$ and $v,w\in \GGG_n$, let $a_1,a_2\in A$ be given by Lemma \ref{lem:partial-spec}: then for any $b\in A \setminus \{a_1,a_2\}$,
we have $vbw \in \GGG$.
\end{lemma}
\begin{proof}
First we must show that $vbw \in \LLL$. Arguing as in the proof of Lemma \ref{lem:GGGG}, observe that if $s=t^\alpha$ is a subword of $vbw$, then either:
\begin{itemize}
\item  $s$ is a subword of $v$ or $w$, in which case $\alpha < \beta$ since $v,w\in \LLL$; or
\item each of $v$ and $w$ accounts for at most $4|t|-1$ of the indices in $s$, in which case $|s| \leq 2(4|t|-1) + 1 = 8|t| - 1 < 8|t|$, implying that $\alpha < 8 \leq \beta$.
\end{itemize}
Recalling Remark \ref{rmk:G4}, to show that $vbw \in \GGG$, it suffices to prove that $\pre_\ell(vbw) < 4$ and $\suf_\ell(vbw) < 4$.
For every $\ell \leq \lceil n/2 \rceil$, we have $\pre_\ell(vbw) = \pre_\ell(v) < 4$ and $\suf_\ell(vbw) = \suf_\ell(w) < 4$ by Lemma \ref{lem:partial-spec}. Moreover, if $\ell > \lceil n/2 \rceil$, then $\ell \geq \lceil n/2 \rceil + 1 \geq (n+2)/2$, so $4\ell \geq 2n + 4 > |vbw|$, and thus $\pre_\ell(vbw) < 4$ and $\suf_\ell(vbw) < 4$. We conclude that $vbw \in \GGG$, as claimed.
\end{proof}

Lemma \ref{lem:GbG} does not establish property \ref{spec-same} when $d=2$, since in this case the set $A \setminus \{a_1,a_2\}$ might be empty. Thus we need to allow the use of a $2$-symbol word in order to ``disrupt the shorter periods''. This extra symbol requires us to exclude the case $\beta = 8$.

\begin{lemma}\label{lem:GuG}
Given $d=2$ and $\beta > 8$, the collection $\GGG$ satisfies \ref{spec-same} with $\tau = 2$.
Indeed, given $n\in \NN$ and $v,w\in \GGG_n$, let $a_1,a_2\in A$ be given by Lemma \ref{lem:partial-spec}: then for any $u\in A^2$ satisfying $u_1 \neq a_1$ and $u_2 \neq a_2$, we have $vuw \in \GGG$.
\end{lemma}
\begin{proof}
We mimic the proof of Lemma \ref{lem:GbG}, with minor modifications. To show that $vuw \in \LLL$, observe that if $s=t^\alpha$ is a subword of $vuw$, then either:
\begin{itemize}
\item  $s$ is a subword of $v$ or $w$, in which case $\alpha < \beta$ since $v,w\in \LLL$; or
\item each of $v$ and $w$ accounts for at most $4|t|-1$ of the indices in $s$, in which case $|s| \leq 2(4|t|-1) + 2 = 8|t|$, implying that $\alpha \leq 8 < \beta$.
\end{itemize}
So it suffices to prove that $\pre_\ell(vuw) < 4$ and $\suf_\ell(vuw) < 4$, which follows as before: given $\ell \leq \lceil n/2 \rceil$, Lemma \ref{lem:partial-spec} gives $\pre_\ell(vuw) = \pre_\ell(v) < 4$ and $\suf_\ell(vuw) = \suf_\ell(w) < 4$; and given $\ell > \lceil n/2 \rceil$, we have $\ell \geq \lceil n/2 \rceil + 1 \geq (n+2)/2$, so $4\ell \geq 2n + 4 > |vuw|$, and thus $\pre_\ell(vuw) < 4$ and $\suf_\ell(vuw) < 4$, completing the proof.
\end{proof}

Now we extend Lemma \ref{lem:GGGG} to smaller values of $\beta$ by increasing $T$.

\begin{lemma}\label{lem:GpGqGrG}
Given any $d\geq 3$ and $\beta \geq 12$, the collection $\GGG$ satisfies \ref{spec-diff} with $T=1$:
for every $u,v,w,x\in \GGG$, there exist $a,b,c\in A$ such that $uavbwcx \in \LLL$.
If $d=2$ and $\beta > 12$, then $\GGG$ satisfies \ref{spec-diff} with $T=2$.
\end{lemma}
\begin{proof}
First suppose that $d\geq 3$ and $\beta\geq 12$. Given $u,v,w,x\in \GGG$, apply Lemma \ref{lem:partial-spec} to $u,v$ to get $a_1,a_2 \in A$ such that fixing $a\in A \setminus \{a_1,a_2\}$, we have
\begin{align*}
\pre_\ell(uav) = \pre_\ell(u) < 4 &\text{ for every } \ell \in \{1,\dots, \lceil |u|/2\rceil \}, \\
\suf_\ell(uav) = \suf_\ell(v) < 4 &\text{ for every } \ell \in \{1,\dots, \lceil |v|/2\rceil \}.
\end{align*}
Obtain $b = b(v,w)$ and $c = c(w,x)$ similarly. Now suppose $s = t^\alpha$ is a subword of $uavbwcx$. If $s$ is a subword of any of $u,v,w,x$, then $\alpha < \beta$ since each of these words lies in $\LLL$. So we can restrict our attention to the case when $s$ includes either the beginning or the end of every word in $\{u,v,w,x\}$ that it intersects. Let $z$ denote any of these four words that is intersected by $s$: then one of the following occurs:
\begin{itemize}
\item $|t| \leq \lceil|z|/2\rceil$, in which case Lemma \ref{lem:partial-spec} guarantees that $s$ does not extend past both endpoints of $z$; or
\item $|t| > \lceil|z|/2\rceil$, in which case $|t| \geq \frac{|z|}2 + 1$, so $|z| \leq 2(|t|-1)$.
\end{itemize}
The first of these two cases can occur for at most two distinct $z$, and for each such $z$, the word $s$ crosses at most $4|t|-1$ of the symbols in $z$. We conclude that
\[
|s| \leq (4|t|-1) + 1 + (2|t|-2) + 1 + (2|t|-2) + 1 + (4|t|-1)
= 12|t| - 3 < 12|t|,
\]
so $\alpha < 12\leq \beta$, which completes the proof of the lemma when $d\geq 3$. For the case $d=2$, we replace the symbols $a,b,c$ with words $p,q,r\in A^2$ once again provided by Lemma \ref{lem:partial-spec} in just the same way (following the procedure in Lemma \ref{lem:GuG}). The rest of the argument is the same, except that in the final inequality we obtain $|s| \leq 12|t|$ instead of $|s| \leq 12|t|-3$.
\end{proof}

\subsection{An entropy gap}\label{sec:get-gap}

\subsubsection{Upper bound on prefixes and suffixes}

Since $\CCC^p = \CCC^s = \AAA^*$, we must obtain an upper bound for $\#(\AAA^*)_n$ to prove Condition \ref{h-gap}. Every element of $\AAA$ is a $4$-power, so $\#(\AAA^*)_n = 0$ if $n$ is not a multiple of $4$. Given $w\in (\AAA^*)_{4k}$ for some $k\in \NN$, we can write $w = (v^{(1)})^4 \cdots (v^{(\ell)})^4$, where $v^{(1)}, \dots, v^{(\ell)} \in \LLL$ and $\sum_{i=1}^\ell |v^{(i)}| = k$. Writing
\[
\mathbb{J}_\ell := \{ \mathbf{k} = (k_1,\dots, k_\ell) \in \NN^\ell : k_1 + \cdots + k_\ell = k \},
\]
and using the bound $\#\LLL_{k_i} \leq d^{k_i}$, we have
\[
(\AAA^*)_{4k} \subset \bigcup_{\ell=1}^k \bigcup_{\mathbf{k} \in \mathbb{J}_\ell} \prod_{i=1}^\ell \LLL_{k_i}
\quad\Rightarrow\quad
\#(\AAA^*)_{4k} \leq \sum_{\ell=1}^k \sum_{\mathbf{k} \in \mathbb{J}_\ell} \prod_{i=1}^\ell \#\LLL_{k_i}
\leq \sum_{\ell=1}^k (\#\mathbb{J}_\ell) d^k.
\]
There is a bijection between $\mathbb{J}_\ell$ and subsets of $\{1,\dots, k-1\}$ with $\ell-1$ elements, obtained by mapping $\mathbf{k}$ to $\{s_1,\dots, s_{\ell-1}\}$, where $s_j = \sum_{i=1}^j k_i$. Thus $\#\mathbb{J}_\ell = \binom{k-1}{\ell-1}$ and $\sum_{\ell=1}^k \#\mathbb{J}_\ell = 2^{k-1}$, which gives the bound
\begin{equation}\label{eqn:Ak-leq}
\#(\AAA^*)_{4k} \leq \frac 12 (2d)^k
\quad\Rightarrow\quad
h(\AAA^*) := \ulim_{n\to\infty} \frac 1n \log \#(\AAA^*)_{n}
\leq \frac 14 \log(2d).
\end{equation}

\subsubsection{General results on specification and growth}\label{sec:spec-growth}

In order to prove a lower bound on $h(X)$, which we do in \S\ref{sec:get-growth}, we first prove some general results about shift spaces. The first of these is Lemma \ref{lem:general-spec}, which stated that the ``one-step'' specification property \ref{spec-same} implies the ``multi-step'' property \ref{spec-same-2}. Following this, we prove that \ref{spec-same-2} gives a uniform relationship between $\#\GGG_n$ and $e^{nh}$, which can be extended to $\#\LLL_n$ using \ref{decomp} and \ref{PR}.

\begin{proof}[Proof of Lemma \ref{lem:general-spec}]
Given $k\in \NN$ and $\DDD \subset \LLL$, let $\mathbf{S}(k,\DDD)$ be the statement that
\begin{itemize}
\item for every $n\in \NN$ and $w^{(1)},\dots, w^{(k)} \in \GGG_n$, there exist $u^{(1)},\dots, u^{(k-1)} \in \LLL_\tau$ such that $w^{(1)}u^{(1)}w^{(2)}\cdots u^{(k-1)} w^{(k)} \in \DDD$.
\end{itemize}
Then the hypothesis of the lemma is that $\mathbf{S}(2,\GGG)$ is true, and \ref{spec-same-2} is the condition that $\mathbf{S}(k,\LLL)$ is true for every $k\in \NN$. 

First observe that if $\mathbf{S}(k,\LLL)$ is true, then $\mathbf{S}(j,\LLL)$ is true for every $1\leq j\leq k$, since given any $w^{(1)}, \dots, w^{(j)} \in \GGG_n$, we can simply add $k-j$ copies of $w^{(j)}$ to the list, apply $\mathbf{S}(k,\LLL)$ to obtain $w^{(1)} u^{(1)} w^{(2)} \cdots u^{(j-1)} w^{(j)} u^{(j)} w^{(j+1)} \cdots u^{(k-1)} w^{(k)} \in \LLL$, and then truncate to obtain the desired word.

To prove \ref{spec-same-2}, it thus suffices to prove that $\mathbf{S}(2^\ell,\LLL)$ is true for every $\ell\in \NN$. We will prove the stronger statement $\mathbf{S}(2^\ell,\GGG)$ by induction. The base case $\ell=1$ is given. Suppose that $\mathbf{S}(2^\ell,\GGG)$ is true for some $\ell\in \NN$. Given $w^{(1)},\dots, w^{(k)} \in \GGG_n$ for $k = 2^{\ell+1}$, apply $\mathbf{S}(2^\ell,\GGG)$ twice:
\begin{itemize}
\item first, to $w^{(1)},\dots, w^{(2^\ell)}$, getting connecting words $u^{(1)},\dots, u^{(2^\ell-1)} \in \LLL_\tau$ such that $v^{(1)} := w^{(1)} u^{(1)} w^{(2)} \cdots u^{(2^\ell-1)} w^{(2^\ell)} \in \GGG$;
\item second, to $w^{(2^\ell+1)},\dots, w^{(2^{\ell+1})}$, getting $u^{(2^\ell+1)},\dots, u^{(2^{\ell+1}-1)} \in \LLL_\tau$ such that $v^{(2)} := w^{(1)} u^{(1)} w^{(2)} \cdots u^{(2^{\ell+1}-1)} w^{(2^{\ell+1})} \in \GGG$.
\end{itemize}
Then apply $\mathbf{S}(2,\GGG)$ to the words $v^{(1)},v^{(2)} \in \GGG_{2^\ell n + (2^\ell - 1)\tau}$, obtaining $u^{(2^\ell)} \in \LLL_\tau$ such that $v^{(1)} u^{(2^\ell)} v^{(2)} \in \GGG$, which proves that $\mathbf{S}(2^{\ell+1},\GGG)$ is true.
\end{proof}

\begin{lemma}\label{lem:G-leq}
Let $\LLL$ be the language of a shift space, and $h$ its topological entropy.
If $\GGG \subset \LLL$ satisfies \ref{spec-same-2}, then writing $C_1 = e^{\tau h}$, we have
\begin{equation}\label{eqn:G-leq}
\#\GGG_n \leq C_1 e^{nh} \quad\text{for all } n\in \NN.
\end{equation}
\end{lemma}
\begin{proof}
Fix $n\in \NN$. Given $k\in \NN$, use \ref{spec-same-2} to define a map $(\GGG_n)^k \to \LLL_{kn+(k-1)\tau}$ by $(w^{(1)},\dots, w^{(k)}) \mapsto w^{(1)} u^{(1)} w^{(2)} \cdots u^{(k-1)} w^{(k)}$, where the connecting words $u^{(i)} \in \LLL_\tau$ depend on the words $w^{(i)}$. This map is injective, so
\[
\#\LLL_{k(n+\tau)} \geq \#\LLL_{kn+(k-1)\tau}
\geq (\#\GGG_n)^k.
\]
Taking logs and dividing by $k$, we get
\[
\log \#\GGG_n \leq \frac 1k \log \#\LLL_{k(n+\tau)}
= \frac{(n+\tau)}{k(n+\tau)} \log \#\LLL_{k(n+\tau)}
\xrightarrow{k\to\infty} (n+\tau) h.
\]
Taking exponentials proves the lemma.
\end{proof}

For the next lemma, we recall from \ref{PR} that
$Q := (\sum_{i=0}^\infty \#\Cp_i e^{-ih})(\sum_{k=0}^\infty \#\Cs_k e^{-kh})$. By convention, we include the empty word in both $\Cp$ and $\Cs$, so $\#\Cp_0 = \#\Cs_0 = 1$, and one can also write
\begin{equation}\label{eqn:Q}
Q = \bigg( 1 + \sum_{i=1}^\infty \#\Cp_i e^{-ih} \bigg)
\bigg( 1 + \sum_{k=1}^\infty \#\Cs_k e^{-kh} \bigg).
\end{equation}

\begin{lemma}\label{lem:L-leq}
Let $\LLL$ be the language of a shift space, and $h$ its topological entropy.
Suppose $\Cp,\GGG,\Cs \subset \LLL$ satisfy \ref{decomp} and \ref{PR}, so that $Q<\infty$.
If $\GGG$ satisfies \eqref{eqn:G-leq}, then we have
\begin{equation}\label{eqn:L-leq}
\#\LLL_n \leq C_1 Q e^{nh} \text{ for all } n\in \NN,
\end{equation}
\end{lemma}
\begin{proof}
By \ref{decomp}, we have $\LLL_n \subset \bigcup_{i=0}^n \bigcup_{k=0}^{n-i}
\Cp_i \GGG_{n-(i+k)} \Cs_k$, so
\[
\#\LLL_n \leq \sum_{i=0}^n \sum_{k=0}^{n-i} \#\Cp_i \#\GGG_{n-(i+k)} \#\Cs_k
\leq \sum_{i=0}^n \sum_{k=0}^{n-i} \#\Cp_i (C_1 e^{nh} e^{-ih} e^{-kh}) \#\Cs_k.
\qedhere
\]
\end{proof}

\subsubsection{Lower bounds on growth}\label{sec:get-growth}

Returning to the shift spaces $\Xdb$, fix $(\beta,d)$ such that $d\geq 3$ and $\beta \geq 8$, or $d=2$ and $\beta\geq 8^+$.
For these parameter values,
we saw in Lemmas \ref{lem:GbG} and \ref{lem:GuG} that $\GGG$ satisfies \ref{spec-same} with $\tau=1$ when $d\geq 3$, and $\tau=2$ when $d=2$. (We did not prove \ref{spec-diff} for $\beta < 12$, but that will not matter for the results in this section.)
By Lemma \ref{lem:general-spec}, $\GGG$ satisfies \ref{spec-same-2}, so Lemma \ref{lem:G-leq} gives
\begin{equation}\label{eqn:G-leq-d}
\#\GGG_n \leq \begin{cases}
e^{(n+2) h}, &\text{if } d=2, \\
e^{(n+1) h}, &\text{if } d\geq 3.
\end{cases}
\end{equation}
Observe that
\[
A^4 \setminus \GGG_4 = \{ a^4 : a\in A \}
\quad\Rightarrow\quad
\#\GGG_4 = d^4 - d.
\]
Thus when $d\geq 3$, we can use \eqref{eqn:G-leq-d} to get
\begin{equation}\label{eqn:e4h}
d^4 - d \leq \#\GGG_4 \leq e^{4h} e^h \leq d e^{4h}
\quad\Rightarrow\quad
d^3 - 1 \leq e^{4h}.
\end{equation}
Since $d^3 - 1 > 2d$ for all $d\geq 3$, we can deduce from \eqref{eqn:Ak-leq} and \eqref{eqn:e4h} that
\[
h(\AAA^*) \leq \frac 14 \log(2d) < \frac 14 \log(d^3 - 1) \leq h(X),
\]
which verifies \ref{h-gap} when $d\geq 3$. We can also prove the second upper bound in \eqref{eqn:card} using Lemma \ref{lem:L-leq} and the fact that
\eqref{eqn:Ak-leq} and \eqref{eqn:e4h} give
\begin{equation}\label{eqn:sum-3}
\begin{aligned}
\sum_{n=1}^\infty \#(\AAA^*)_n e^{-nh}
&= \sum_{k=1}^\infty \#(\AAA^*)_{4k} e^{-4kh}
\leq \sum_{k=1}^\infty \frac 12 \frac{(2d)^k}{(d^3-1)^k} \\
&= \frac 12 \cdot \frac {\frac{2d}{d^3-1}}{1- \frac{2d}{d^3-1}} 
= \frac{d}{d^3-2d-1}.
\end{aligned}
\end{equation}
It remains to verify \ref{h-gap} when $d=2$. Here, using $\GGG_4$ as in \eqref{eqn:e4h} produces the weaker bound $d^4 - d \leq d^2 e^{4h}$ because of the weaker inequality in \eqref{eqn:G-leq-d}, and it turns out that this is not enough to establish \ref{h-gap}. Thus we work instead with $\GGG_8$, observing that
\[
A^8 \setminus \GGG_8
= \{ a^4 w : a\in A, w\in A^4\}
\cup \{ vb^4 : v\in A^4, b\in A\}
\cup \{u^4 : u\in A^2\},
\]
from which we conclude that
\[
\#\GGG_8 \geq d^8 - 2d^5 - d^2 = 256 - 64 - 4 = 188.
\]
Using this together with \eqref{eqn:G-leq-d} and the fact that $e^{2h} \leq d^2 = 4$ gives
\[
e^{8h} \geq e^{-2h} \#\GGG_8 \geq 47,
\]
so we can establish \ref{h-gap} using \eqref{eqn:Ak-leq} to write
\[
h(\AAA^*) \leq \frac 14 \log(4)
= \frac 18 \log(16) < \frac 18 \log(47) \leq h.
\]
Now the first upper bound in \eqref{eqn:card} follows from Lemma \ref{lem:L-leq} and
\begin{equation}\label{eqn:sum-2}
\sum_{k=1}^\infty \#(\AAA^*)_{4k} e^{-4kh}
\leq \sum_{k=1}^\infty \frac 12 \cdot \frac{4^k}{(47)^{k/2}} 
= \frac 12 \cdot \frac {\frac{4}{\sqrt{47}}}{1- \frac{4}{\sqrt{47}}}
= \frac{2}{\sqrt{47}-4},
\end{equation}
since $(1+ \frac{2}{\sqrt{47}-4})^2 \leq 2.892$.

\section{Proof of the general result}\label{sec:pf-gen}

\subsection{Structure of the proof}\label{sec:structure}

We will prove Theorem \ref{thm:general} via three propositions, which we state below and prove in \S\S\ref{sec:counting}--\ref{sec:uniqueness}. These propositions use different combinations of the conditions \ref{decomp}, \ref{PR}, \ref{spec-diff}, and \ref{spec-same-2}. If all four conditions hold, then all three propositions apply and combine to establish Theorem \ref{thm:general}.

Observe that Lemmas \ref{lem:G-leq} and \ref{lem:L-leq} already proved the bounds on $\#\LLL_n$ claimed in the statement of Theorem \ref{thm:general}. Writing $C_2 := Q e^{\tau h}$, we will record these as:
\begin{equation}\label{eqn:L-bounds}
e^{nh} \leq \#\LLL_n \leq C_2 e^{nh}
\quad\text{for all } n\in \NN.
\end{equation}
Given $i,k,n\in \NN$ such that $i+k \leq n$, define a map $\ph_{i,k} \colon \LLL_n \to \LLL_{n-(i+k)}$ by
\begin{equation}\label{eqn:phi}
\ph_{i,k}(w) = w_{i+1} w_{i+2} \cdots w_{n-k-1}.
\end{equation}
That is, $\ph_{i,k}$ truncates $w$ by removing the first $i$ and last $k$ symbols. Also, given a collection of words $\DDD_n \subset \LLL_n$, we will write $[\DDD_n] = \bigcup_{w\in \DDD_n} [w] \subset X$ for the union of the corresponding cylinders.

\begin{proposition}\label{prop:good-lower}
Let $X$ be a shift space with topological entropy $h$ whose language $\LLL$ contains collections $\Cp,\GGG,\Cs$ satisfying \ref{decomp} and \ref{PR}, and for which \eqref{eqn:L-bounds} holds. 

Then there exist $\theta>0$ and $M\in \NN$ such that given any collection of words $\DDD_n \subset \LLL_n$ satisfying $\nu[\DDD_n] \geq \frac 12$ for some $n > 2M$ and some MME $\nu$, there exist $i,k \in \{0,1,\dots, M\}$ such that 
\begin{equation}\label{eqn:phi-geq}
\#\big(\ph_{i,k}(\DDD_n) \cap \GGG\big) \geq \theta e^{nh}.
\end{equation}
In particular, there exists $j\in [n-2M, n] \cap \NN$ such that
\begin{equation}\label{eqn:G-geq}
\#\GGG_j \geq \theta e^{nh}.
\end{equation}
\end{proposition}

The next proposition does not require \ref{decomp} or \ref{PR}, but does require \ref{spec-diff}.

\begin{proposition}\label{prop:Gibbs}
Let $X$ be a shift space with topological entropy $h$ whose language $\LLL$ satisfies \eqref{eqn:L-bounds} and contains a collection $\GGG$ satisfying \ref{spec-diff}. Suppose that $\GGG$ satisfies \eqref{eqn:G-geq} in the following sense:
there are $\theta>0$ and $M \in \NN$ such that for every $n\in \NN$, there exists  $j\in [n-2M, n] \cap \NN$ such that $\#\GGG_j \geq \theta e^{nh}$.

Then there exists an MME $\mu$ on $X$ that has the following Gibbs-type property: there is $c>0$ such that
\begin{equation}\label{eqn:Gibbs}
\mu[w] \geq ce^{-|w| h} \text{ for all } w\in \GGG,
\end{equation}
and such that with $T\in \NN$ as in \ref{spec-diff}, we have
\begin{equation}\label{eqn:Gibbs-2}
\mu([u] \cap \sigma^{-(|u| + T)}[v]) \geq c e^{-(|u| + |v|) h} \text{ for all } u,v\in \GGG.
\end{equation}
\end{proposition}

\begin{proposition}\label{prop:unique}
Let $X$ be a shift space with topological entropy $h$ whose language $\LLL$ contains a collection $\GGG$ satisfying the conclusion of Proposition \ref{prop:good-lower}, and suppose that $X$ admits an MME $\mu$ that satisfies the two-step Gibbs bound \eqref{eqn:Gibbs-2}. Then $\mu$ is the unique MME.
\end{proposition}

\subsection{Uniform counting bounds}\label{sec:counting}

In this section, we prove Proposition \ref{prop:good-lower}. 

\begin{proposition}\label{prop:unif-Katok}
Let $X$ be a shift space with topological entropy $h$ whose language $\LLL$ satisfies \eqref{eqn:L-bounds}.
Then for all $\gamma \in (0,1)$, there exists $\alpha>0$ such that for every $\DDD_n \subset \LLL_n$ satisfying $\nu([\DDD_n]) \geq \gamma$ for some MME $\nu$, we have $\#\DDD_n \geq \alpha e^{nh}$.
\end{proposition}
\begin{proof}
We use the following general facts about probability vectors \cite[Corollary 4.2.1 and Theorem 4.3(ii)]{pW82}: first, if $\mathbf{p} := (p_1,\dots, p_k)$ is a probability vector, then
\begin{equation}\label{eqn:h-leq}
H(\mathbf{p}) := \sum_{i=1}^k -p_i \log p_i \leq \log k;
\end{equation} 
and second, if given $I \subset \{1,\dots, k\}$ we write $p_I := \sum_{i\in I} p_i$ and $\mathbf{p}_I = (p_i / p_I)_{i\in I}$ for the corresponding probability vector of conditional probabilities, then we have
\begin{equation}\label{eqn:split-h}
H(\mathbf{p}) = H(p_I, p_{I^c})
+ p_I H(\mathbf{p}_I) + p_{I^c} H(\mathbf{p}_{I^c}),
\end{equation}
where $I^c = \{1,\dots, k\} \setminus I$.
In the setting of the proposition, the partition into $n$-cylinders is generating for $\sigma^n$, so we have
\begin{align*}
nh &= nh_\nu(\sigma) = h_\nu(\sigma^n)
\leq \sum_{w\in \LLL_n} -\nu[w] \log\nu[w] \\
&\leq \log 2 + \nu[\DDD_n] \sum_{w\in \DDD_n} -\frac{\nu[w]}{\nu[\DDD_n]} \log \frac{\nu[w]}{\nu[\DDD_n]}
+ \nu[\DDD_n^c] \sum_{w\in \DDD_n^c} -\frac{\nu[w]}{\nu[\DDD_n^c]} \log \frac{\nu[w]}{\nu[\DDD_n^c]},
\end{align*}
where the last inequality uses \eqref{eqn:split-h} and \eqref{eqn:h-leq} with $k=2$. Applying \eqref{eqn:h-leq} to the two sums and using the fact that $\#\DDD_n^c \leq \#\LLL_n$, we get
\begin{align*}
nh &\leq \log 2 + \nu[\DDD_n] \log \#\DDD_n + (1-\nu[\DDD_n]) \log \#\LLL_n \\
&\leq \log 2 + \nu[\DDD_n] \log \#\DDD_n + nh + \log C_2 - \nu[\DDD_n] \log \#\LLL_n,
\end{align*}
where the last inequality uses \eqref{eqn:L-bounds}.
This implies that
\[
-\log 2 - \log C_2 \leq \nu[\DDD_n] \log (\#\DDD_n / \#\LLL_n)
\leq \gamma \log (\#\DDD_n / \#\LLL_n),
\]
since $\nu[\DDD_n] \geq \gamma$ and the logarithm is nonpositive. We obtain
\[
(\#\DDD_n/\#\LLL_n)^\gamma \geq (2C_2)^{-1}
\quad\Rightarrow\quad
\#\DDD_n \geq (2C_2)^{-1/\gamma} \#\LLL_n,
\]
and since $\#\LLL_n \geq e^{nh}$ by \eqref{eqn:L-bounds}, this proves Proposition \ref{prop:unif-Katok} with $\alpha = (2C_2)^{-1/\gamma}$.
\end{proof}

Taking $\alpha$ corresponding to $\gamma = \frac 12$ in Proposition \ref{prop:unif-Katok}, the proof of Proposition \ref{prop:good-lower} is completed by the following. 

\begin{lemma}\label{lem:alpha-theta}
Let $X$ be a shift space with topological entropy $h$ whose language $\LLL$ contains collections $\Cp,\GGG,\Cs$ satisfying \ref{decomp} and \ref{PR}, and for which there exists $C_1>0$ such that $\#\GGG_j \leq C_1 e^{jh}$ for all $j\in\NN$.

Then for every $\alpha \in (0,1)$, there exist $M\in \NN$ and $\theta\in (0,1)$ such that given any $n> 2M$ and $\DDD_n \subset \LLL_n$ satisfying $\#\DDD_n \geq \alpha e^{nh}$, there exist $i,k \in \{0,1,\dots, M\}$ such that
\eqref{eqn:phi-geq} holds: $\#(\ph_{i,k}(\DDD_n) \cap \GGG) \geq \theta e^{nh}$.
\end{lemma}
\begin{proof}
Given $M\in \NN$, let $\mathbb{I}(M) := \{ (i,k) \in (\NN \cup \{0\})^2 : \max(i,k) > M\}$, and 
observe that by \ref{PR}, we have $\delta(M) := \sum_{(i,k) \in \mathbb{I}(M)} (\#\CCC_i^p e^{-ih})
(\#\Cs_k e^{-kh}) \to 0$ as $M \to\infty$.
Also, consider for each $n>2M$ the quantity
\begin{equation}\label{eqn:rnM}
r_n(M) := \max \{ \#(\ph_{i,k}(\DDD_n) \cap \GGG) : 0\leq i,k\leq M \}.
\end{equation}
We must prove that there exists $M\in \NN$ and $\theta>0$ such that $r_n(M) \geq \theta e^{nh}$ for all $n$.
As in the proof of Lemma \ref{lem:L-leq}, we use \ref{decomp} to get
\[
\DDD_n \subset \bigcup_{i=0}^n \bigcup_{k=0}^{n-i}
\Cp_i \big( \ph_{i,k}(\DDD_n) \cap \GGG\big) \Cs_k.
\]
Writing $\mathbb{I}_n(M) := \{ (i,k) \in \mathbb{I}(M) : i+k \leq n\}$
and using the upper bound on $\#\GGG_j$,
we obtain for each $M\in \NN$ the inequality
\begin{align*}
\#\DDD_n &\leq \sum_{i=0}^M \sum_{k=0}^M (\#\Cp_i) \#(\ph_{i,k}(\DDD_n) \cap \GGG)(\#\Cs_k)
+ \sum_{(i,k) \in \mathbb{I}_n(M)} (\#\Cp_i)(\#\GGG_{n-(i+k)})(\#\Cs_k) \\
&\leq \sum_{i=0}^M \sum_{k=0}^M (\#\Cp_i)(\#\Cs_k) r_n(M)
+ \sum_{(i,k) \in \mathbb{I}_n(M)} (\#\Cp_i) C_1 e^{nh} e^{-ih} e^{-kh} (\#\Cs_k).
\end{align*}
Writing $S(M) = (\sum_{i=0}^M \#\Cp_i)(\sum_{k=0}^M \#\Cs_k)$, we use the inequality $\#\DDD_n \geq \alpha e^{nh}$ to get
\[
\alpha e^{nh} \leq  S(M) r_n(M) + \delta(M) C_1 e^{nh}.
\]
Since $\delta(M) \to 0$ as $M\to\infty$, we can take $M$ sufficiently large that $\delta(M) C_1 < \alpha$, and obtain
\[
r_n(M) \geq S(M)^{-1}(\alpha - \delta(M) C_1) e^{nh},
\]
which proves the lemma.
\end{proof}

\subsection{A Gibbs-type measure}\label{sec:Gibbs}

In this section, we prove Proposition \ref{prop:Gibbs} by constructing an MME that has the Gibbs-type properties \eqref{eqn:Gibbs} and \eqref{eqn:Gibbs-2}. Note that the first of these is a special case of the second, by taking $v$ to be the empty word, so we will only give the proof for \eqref{eqn:Gibbs-2}.

The construction of $\mu$, as well as the proof that it is invariant and is an MME, follows the usual argument from Misiurewicz's proof of the variational principle:
for each $n\in \NN$, let $\nu_n$ be a Borel probability measure on $X$ with the property that $\nu_n[w] = 1/\#\LLL_n$ for every $w\in \LLL_n$; then let $\mu_n := \frac 1n \sum_{k=0}^{n-1} \sigma_*^k \nu_n$, and let $\mu$ be any weak*-limit point of the sequence $\mu_n$. See \cite[Theorem 8.6]{pW82} for the argument that $\mu$ is invariant and is an MME.

It remains to prove \eqref{eqn:Gibbs-2}. The following argument can be extracted from \cite[Lemma 5.15]{CT12}, but since the hypotheses are stated differently here, we give the details. Figure \ref{fig:Gibbs} illustrates the bookkeping involved in the following computations by showing which indices are included in which words.

\begin{figure}[htbp]
\begin{tikzpicture}
\draw (0,1.5)--(0,-.3) node[below]{$0$};
\draw (14,1.5)--(14,-.3) node[below]{$n$};
\draw[fill=red!25!white] (4,0) rectangle (6,1.2);
\node at (5,0.6) {$u$};
\draw[fill=blue!25!white] (7,0) rectangle (10,1.2);
\node at (8.5,0.6) {$v$};
\draw[fill=yellow!25!white] (0.6,0) rectangle (3,0.8);
\node at (1.8,0.4) {$w'$};
\draw[fill=yellow!25!white] (11,0) rectangle (13.3,0.8);
\node at (12.15,0.4) {$x'$};
\draw[fill=black!20!white] (0,0) rectangle (0.6,0.4);
\node at (0.3,0.2){$y$};
\draw[fill=black!20!white] (3,0) rectangle (4,0.4);
\node at (3.5,0.2){$q$};
\draw[fill=black!20!white] (6,0) rectangle (7,0.4);
\node at (6.5,0.2){$p$};
\draw[fill=black!20!white] (10,0) rectangle (11,0.4);
\node at (10.5,0.2){$r$};
\draw[fill=black!20!white] (13.3,0) rectangle (14,0.4);
\node at (13.65,0.2){$z$};
\draw[fill=green!25!white] (0,0.8) rectangle (4,1.2);
\node at (2,1){$w$};
\draw[fill=green!25!white] (10,0.8) rectangle (14,1.2);
\node at (12,1){$x$};
\draw (4,1.5)--(4,-0.3);
\draw (6,1.5)--(6,-0.3);
\draw (7,1.5)--(7,-0.3);
\draw (10,1.5)--(10,-0.3);
\draw[<->,red] (.1,1.4) -- (3.9,1.4);
\draw[red] (2,1.4) node[above] {$k$};
\draw[<->,blue] (6.1,1.4) -- (6.9,1.4);
\draw[blue] (6.5,1.4) node[above] {$T$};
\draw[<->,red] (10.1,1.4) -- (13.9,1.4);
\draw[red] (12,1.4) node[above] {$j$};
\draw[<->,red] (0.7, -0.2) -- (2.9, -0.2);
\draw[red] (1.8,-0.2) node[below]{$k'$};
\draw[<->,red] (11.1, -0.2) -- (13.2, -0.2);
\draw[red] (12.15,-0.2) node[below]{$j'$};
\draw[<->,blue] (3.1, -0.2) -- (3.9, -0.2);
\draw[blue] (3.5,-0.2) node[below]{$T$};
\draw[<->,blue] (10.1, -0.2) -- (10.9, -0.2);
\draw[blue] (10.5,-0.2) node[below]{$T$};
\end{tikzpicture}
\caption{Getting the Gibbs bound from property \ref{spec-diff}.}
\label{fig:Gibbs}
\end{figure}

Observe that given $u,v\in \GGG$ and $n\in \NN$, for each $0 \leq k < n-|u|-T-|v|$, we can write $j = n - k - |u| - T - |v|$ and obtain
\begin{equation}\label{eqn:sigma*geq}
\sigma_*^k \nu_n ([u] \cap \sigma^{-(|u|+T)}[v])
= \frac{ \#\{ (w,p,x) \in \LLL_k \times \LLL_T\times \LLL_j : wupvx \in \LLL \}}
{\#\LLL_n}.
\end{equation}
If $k,j\geq 2M$, then \eqref{eqn:G-geq} gives  $k' \in [k-T -2M, k-T]$ and $j' \in [j-T-2M, j-T]$ such that
\begin{equation}\label{eqn:Gj'}
\#\GGG_{k'} \geq \theta e^{(k-T)h}
\quad\text{and}\quad
\#\GGG_{j'} \geq \theta e^{(j-T)h}
\end{equation}
For every $w'\in \GGG_{k'}$ and $x' \in \GGG_{j'}$, \ref{spec-diff} gives $q,p,r \in \LLL_T$ such that $w'qupvrx' \in \LLL$. Thus there exist $y,z\in \LLL$ such that
\[
w := yw'q \in \LLL_k,\quad
x := rx'z \in \LLL_j,\quad\text{and}\quad
wupvx = yw'qupvrx'z \in \LLL_n.
\]
Together with \eqref{eqn:sigma*geq}, \eqref{eqn:Gj'}, and the bound $\#\LLL_n \leq C_2 e^{nh}$ from \eqref{eqn:L-bounds}, this gives
\[
\sigma_*^k \nu_n ([u] \cap \sigma^{-(|u|+T)}[v])
\geq \theta^2 e^{-2Th} e^{jh+kh} C_2^{-1} e^{-nh}
= c e^{-(|u|+|v|)h},
\]
where $c := \theta^2 e^{-3Th} C_2^{-1}$.
This bound holds for all $2M \leq k < n - |u| - T - |v| - 2M$, so upon averaging over all $0\leq k < n$, we get
\[
\mu_n([u] \cap \sigma^{-(|u|+T)}[v])
\geq \frac{n-|u|-T-|v|-4M}{n} \cdot c e^{-(|u|+|v|)h}.
\]
Sending $n\to\infty$ along an appropriate subsequence proves \eqref{eqn:Gibbs-2}.

\subsection{Uniqueness}\label{sec:uniqueness}

In this section, we prove Proposition \ref{prop:unique}. 
This is where the new ideas from \cite{PYY22} become essential; the arguments in \S\ref{sec:counting} and \S\ref{sec:Gibbs} closely mirror the ones in \cite{CT12} (although the map $\ph_{i,k}$ and the estimate \eqref{eqn:phi-geq} did not appear there), 
but to prove that $\mu$ is ergodic and that there is no other MME, the arguments in \cite{CT12} required an extra condition. Thus we now follow \cite{PYY22}, although with substantial simplifications, since that paper is written for flows on compact metric spaces and deals with equilibrium states for a broader class of potential functions.

The key innovation leading to the improvement over \cite{CT12} is the next lemma, which follows \cite[Lemma 7.3]{PYY22}, and 
holds in any shift space.

\begin{lemma}\label{lem:approx}
Given any mutually singular $\nu_1,\nu_2 \in \MMM_\sigma(X)$, and any $\zeta>0$, there is a collection of words $\UUU \subset \LLL$ and $n_0\in \NN$ such that
\begin{enumerate}
\item $\nu_1([\UUU_n]) \geq 1-\zeta$ for all $n$, and
\item $\nu_2([\ph_{i,j} \UUU_n]) \leq \zeta$ for all $i,j,n$ satisfying $n \geq 2(n_0 + \max(i,j))$.
\end{enumerate}
\end{lemma}
\begin{proof}
Since $\nu_1 \perp \nu_2$, there are disjoint Borel sets $P_1,P_2 \subset X$ such that $\nu_\ell(P_\ell)=1$ for $\ell=1,2$. Fix compact sets $K_\ell \subset P_\ell$ such that $\nu_\ell(K_\ell) \geq 1-\zeta$. The distance between $K_1$ and $K_2$ is positive, so there exists $n_0\in \NN$ such that given any $u\in \LLL_{2n_0+1}$, the centered cylinder $\sigma^{n_0} [u]$ intersects at most one of $K_1$ and $K_2$.

\begin{figure}[htbp]
\begin{tikzpicture}
\draw[fill=blue!25!white] (0,0.8) rectangle (10,1.2);
\node at (5.5,1){$w$};
\draw[fill=green!25!white] (1.5,0.4) rectangle (9,0.8);
\node at (5.5,0.6){$v$};
\draw[fill=yellow!25!white] (3,0) rectangle (7,0.4);
\node at (5.5,0.2){$u$};
\draw (0,1.5)--(0,-.2) node[below]{$0$};
\draw (10,1.5)--(10,-.2) node[below]{$n$};
\draw (5,1.5)--(5,-.2) node[below]{$k$};
\draw[<->,red] (0.1, 0.2) -- (1.4, 0.2);
\draw[red] (0.75,0.2) node[below]{$i$};
\draw[<->,red] (9.1, 0.2) -- (9.9, 0.2);
\draw[red] (9.5,0.2) node[below]{$j$};
\draw[<->,red] (3.1, -0.2) -- (4.9, -0.2);
\draw[red] (4,-0.2) node[below]{$n_0$};
\draw[<->,red] (5.1, -0.2) -- (6.9, -0.2);
\draw[red] (6,-0.2) node[below]{$n_0$};
\end{tikzpicture}
\caption{Words in $\UUU_n$ miss $K_2$.}
\label{fig:wvu}
\end{figure}

We claim that $\UUU_n = \{ w \in \LLL_n : [w] \cap \sigma^{-\lfloor n/2 \rfloor} K_1 \neq \emptyset \}$ satisfies the conclusions of the lemma. Clearly $[\UUU_n] \supset \sigma^{-\lfloor n/2 \rfloor}K_1$, and since $\nu_1$ is invariant we get
\[
\nu_1[\UUU_n] \geq \nu_1(\sigma^{-\lfloor n/2 \rfloor} K_1) = \nu_1(K_1) \geq 1-\zeta,
\]
which proves the first claim. For the second claim, 
we write $k = \lfloor \frac n2 \rfloor$ and start by showing that
given any $i,j \leq k - n_0$ and any $w\in \UUU_n$, we have
\begin{equation}\label{eqn:miss-K2}
\sigma^{k-i} [\ph_{i,j} (w)] \cap K_2 = \emptyset.
\end{equation}
Writing $v = \ph_{i,j}(w)$ and $u = w_{[k-n_0,k+n_0]}$, as shown in Figure \ref{fig:wvu}, we have
\[
\sigma^k[w] \subset \sigma^{k-i}[v] \subset \sigma^{n_0} [u].
\]
Since $\sigma^k[w]$ intersects $K_1$, we see that $\sigma^{n_0}[u]$ does as well, so $\sigma^{n_0}[u] \cap K_2 = \emptyset$ by our choice of $n_0$.
This proves \eqref{eqn:miss-K2},
which gives $\sigma^{k-i}[\ph_{i,j}\UUU_n] \subset X \setminus K_2$. By our choice of $K_2$ and invariance of $\nu_2$, we now have
\[
\nu_2([\ph_{i,j}\UUU_n]) = \nu_2(\sigma^{k-i}[\ph_{i,j}\UUU_n])
\leq \nu_2(X\setminus K_2) < \zeta.\qedhere
\]
\end{proof}

Now we prove uniqueness of the MME $\mu$ by showing that there is no MME $\nu\perp \mu$, and that $\mu$ is ergodic (so the only MME $\nu\ll\mu$ is $\mu$ itself). The first of these follows \cite[Proposition 7.2]{PYY22}. Recall that $c$ is the constant in the lower Gibbs bounds \eqref{eqn:Gibbs} and \eqref{eqn:Gibbs-2}, and $\theta,M$ are given by Proposition \ref{prop:good-lower}.

\begin{proposition}\label{prop:no-perp}
Under the hypotheses of Proposition \ref{prop:unique}, there is no MME $\nu\perp \mu$.
\end{proposition}
\begin{proof}
Suppose for a contradiction that $\nu\perp\mu$ is an  MME. 
Fix $\zeta < \min(\frac 12, c\theta)$, and apply Lemma \ref{lem:approx} with $\nu_1 = \nu$ and $\nu_2 = \mu$ to get $\UUU\subset \LLL$ and $n_0\in\NN$ such that $\nu([\UUU_n]) \geq 1-\zeta$ for all $n$, and
\begin{equation}\label{eqn:leq-beta}
\mu([\ph_{i,j} \UUU_n]) \leq \zeta
\end{equation}
for all $i,j,n$ satisfying $n \geq 2(n_0 + \max(i,j))$.

Since $\zeta<\frac 12$, we have $\nu([\UUU_n]) > \frac 12$ for all $n$, and since $\nu$ is an MME, Proposition \ref{prop:good-lower} gives  $i,j\in \{0,1,\dots, M\}$ (depending on $n$) with
\[
\#(\ph_{i,j}(\UUU_n) \cap \GGG) \geq \theta e^{nh};
\]
thus the Gibbs property \eqref{eqn:Gibbs} on $\GGG$ gives
\[
\mu([\ph_{i,j} (\UUU_n) \cap \GGG])
\geq c e^{-(n-(i+j)) h} \#(\ph_{i,j}(\UUU_n) \cap \GGG)
\geq c e^{(i+j) h} \theta \geq c\theta.
\]
For all $n\geq 2(n_0 + M)$ and this choice of $i,j$,
we also have \eqref{eqn:leq-beta}, which leads to
\[
\zeta \geq \mu([\ph_{i,j} \UUU_n])
\geq \mu([\ph_{i,j} (\UUU_n) \cap \GGG]) \geq c\theta.
\]
This contradicts our choice of $\zeta$ and completes the proof of Proposition \ref{prop:no-perp}.
\end{proof}

The proof of ergodicity follows \cite[Proposition 7.4]{PYY22}:

\begin{proposition}\label{prop:ergodic}
Under the hypotheses of Proposition \ref{prop:unique}, $\mu$ is ergodic.
\end{proposition}
\begin{proof}
Suppose for a contradiction that $P\subset X$ is invariant and $0 < \mu(P) < 1$.
Let $\nu$ and $\nu'$ be the normalizations of $\mu$ restricted to $P$ and $P^c$, respectively.
Fix $\zeta < \min (\frac 12, \frac 12 c\theta^2)$, and apply Lemma \ref{lem:approx} twice -- once with $\nu_1=\nu$ and $\nu_2=\nu'$, and once with the roles reversed -- to obtain $\UUU,\UUU' \subset \LLL$ and $n_0 \in \NN$ such that $\nu([\UUU_n]) \geq 1-\zeta$ and $\nu'([\UUU'_n]) \geq 1-\zeta$ for every $n$, and moreover
\begin{equation}\label{eqn:nunu'}
\nu([\ph_{i,j}\UUU_n']) \leq \zeta \text{ and } \nu'([\ph_{i,j}\UUU_n]) \leq \zeta
\end{equation}
for all $i,j,n$ satisfying $n \geq 2(n_0 + \max(i,j))$.

Each of $\nu$ and $\nu'$ is an MME since entropy depends affinely on the measure, and the MME $\mu$ is a convex combination of $\nu$ and $\nu'$. Thus we can apply Proposition \ref{prop:good-lower} to both $\UUU$ and $\UUU'$, obtaining  $i,j,i',j' \in \{0,1,\dots, M\}$ (depending on $n$) such that 
\[
\VVV(n) := \ph_{i,j}(\UUU_n) \cap \GGG \subset \LLL_{n-(i+j)}
\text{ and } \VVV'(n) := \ph_{i',j'}(\UUU_n')\cap \GGG \subset \LLL_{n-(i'+j')}
\]
both have cardinality at least $\theta e^{nh}$.

Let $k = n-(i+j) + T$, so that $|v| + T = k$ for every $v\in \VVV(n)$. Given any such $v$, and any $v' \in \VVV'(n)$, the ``two-step'' Gibbs property in \eqref{eqn:Gibbs-2} gives
\[
\mu([v] \cap \sigma^{-k}[v']) \geq c e^{-(|v| + |v'|)h} \geq c e^{-2nh}.
\]
Summing over all such $v$ and $v'$ gives
\begin{equation}\label{eqn:mugeq}
\mu([\VVV(n)] \cap \sigma^{-k} [\VVV'(n)])
\geq c e^{-2nh} (\#\VVV(n)) (\#\VVV'(n)) \geq c \theta^2.
\end{equation}
We will bound this measure above, for a contradiction.
Given $x\in [\VVV(n)] \cap \sigma^{-k}[\VVV'(n)]$, we either have $x\in P^c$ or $x\in P = \sigma^{-k} P$ (since $P$ is invariant), and thus
\begin{equation}\label{eqn:V-kV}
[\VVV(n)] \cap \sigma^{-k}[\VVV'(n)]
\subset ([\VVV(n)] \cap P^c) \cup \sigma^{-k} ([\VVV'(n)] \cap P).
\end{equation}
Observe that whenever $n\geq 2(n_0 + M)$, \eqref{eqn:nunu'} gives
\[
\mu([\VVV(n)] \cap P^c) = \nu'([\VVV(n)]) \mu(P^c) \leq 
\nu'([\ph_{i,j}(\UUU_n)]) \leq \zeta
\]
and similarly
\[
\mu(\sigma^{-k}([\VVV'(n)] \cap P)) = \mu([\VVV'(n)] \cap P) \leq \zeta,
\]
so that by \eqref{eqn:V-kV}, we have
\[
\mu([\VVV(n)] \cap \sigma^{-k}[\VVV'(n)]) \leq 2\zeta.
\]
Comparing this with \eqref{eqn:mugeq} gives $2\zeta \geq c\theta^2$, contradicting our choice of $\zeta$ and completing the proof of Proposition \ref{prop:ergodic}.
\end{proof}

\bibliographystyle{amsalpha}
\bibliography{comb-words,references}

\providecommand{\bysame}{\leavevmode\hbox to3em{\hrulefill}\thinspace}
\providecommand{\MR}{\relax\ifhmode\unskip\space\fi MR }
\providecommand{\MRhref}[2]{%
  \href{http://www.ams.org/mathscinet-getitem?mr=#1}{#2}
}
\providecommand{\href}[2]{#2}
\begin{thebibliography}{BCFT18}

\bibitem[AS03]{AS03}
Jean-Paul Allouche and Jeffrey Shallit, \emph{Automatic sequences}, Cambridge
  University Press, Cambridge, 2003, Theory, applications, generalizations.
  \MR{1997038}

\bibitem[BCFT18]{BCFT18}
K.~Burns, V.~Climenhaga, T.~Fisher, and D.~J. Thompson, \emph{Unique
  equilibrium states for geodesic flows in nonpositive curvature}, Geom. Funct.
  Anal. \textbf{28} (2018), no.~5, 1209--1259. \MR{3856792}

\bibitem[Ber88]{aB88}
Anne Bertrand, \emph{Specification, synchronisation, average length}, Coding
  theory and applications ({C}achan, 1986), Lecture Notes in Comput. Sci., vol.
  311, Springer, Berlin, 1988, pp.~86--95. \MR{960710}

\bibitem[Ber05]{jB05}
Jean Berstel, \emph{Growth of repetition-free words---a review}, Theoret.
  Comput. Sci. \textbf{340} (2005), no.~2, 280--290. \MR{2150766}

\bibitem[Bow75a]{rB-Bern}
Rufus Bowen, \emph{Bernoulli equilibrium states for {A}xiom {A}
  diffeomorphisms}, Math. Systems Theory \textbf{8} (1974/75), no.~4, 289--294.
  \MR{387539}

\bibitem[Bow75b]{Bow75}
\bysame, \emph{Some systems with unique equilibrium states}, Math. Systems
  Theory \textbf{8} (1974/75), no.~3, 193--202. \MR{399413}

\bibitem[BR25]{BR25}
Vuong Bui and Matthieu Rosenfeld, \emph{An explicit condition for boundedly
  supermultiplicative subshifts}, Comb. Theory \textbf{5} (2025), no.~3, Paper
  No. 11, 20. \MR{4962125}

\bibitem[Bra83]{fB83}
Franz-Josef Brandenburg, \emph{Uniformly growing {$k$}th power-free
  homomorphisms}, Theoret. Comput. Sci. \textbf{23} (1983), no.~1, 69--82.
  \MR{693069}

\bibitem[Car07]{aC07}
Arturo Carpi, \emph{On {D}ejean's conjecture over large alphabets}, Theoret.
  Comput. Sci. \textbf{385} (2007), no.~1-3, 137--151. \MR{2356248}

\bibitem[CKX12]{CKX12}
Ehsan Chiniforooshan, Lila Kari, and Zhi Xu, \emph{Pseudopower avoidance},
  Fund. Inform. \textbf{114} (2012), no.~1, 55--72. \MR{2953324}

\bibitem[Cli18]{Cli18}
Vaughn Climenhaga, \emph{Specification and towers in shift spaces}, Comm. Math.
  Phys. \textbf{364} (2018), no.~2, 441--504. \MR{3869435}

\bibitem[CMR20]{CMR20}
James~D. Currie, Lucas Mol, and Narad Rampersad, \emph{The number of threshold
  words on {$n$} letters grows exponentially for every {$n \geq 27$}}, J.
  Integer Seq. \textbf{23} (2020), no.~3, Art. 20.3.1, 12. \MR{4105860}

\bibitem[CR11]{CR11}
James Currie and Narad Rampersad, \emph{A proof of {D}ejean's conjecture},
  Math. Comp. \textbf{80} (2011), no.~274, 1063--1070. \MR{2772111}

\bibitem[CT12]{CT12}
Vaughn Climenhaga and Daniel~J. Thompson, \emph{Intrinsic ergodicity beyond
  specification: {$\beta$}-shifts, {$S$}-gap shifts, and their factors}, Israel
  J. Math. \textbf{192} (2012), no.~2, 785--817. \MR{3009742}

\bibitem[CT13]{CT13}
\bysame, \emph{Equilibrium states beyond specification and the {B}owen
  property}, J. Lond. Math. Soc. (2) \textbf{87} (2013), no.~2, 401--427.
  \MR{3046278}

\bibitem[CT21]{CT21}
\bysame, \emph{Beyond {B}owen's specification property}, Thermodynamic
  formalism, Lecture Notes in Math., vol. 2290, Springer, Cham, [2021]
  \copyright 2021, pp.~3--82. \MR{4436821}

\bibitem[Dao13]{yD13}
Yair Daon, \emph{Bernoullicity of equilibrium measures on countable {M}arkov
  shifts}, Discrete Contin. Dyn. Syst. \textbf{33} (2013), no.~9, 4003--4015.
  \MR{3038050}

\bibitem[Dek76]{fD76}
F.~M. Dekking, \emph{On repetitions of blocks in binary sequences}, J.
  Combinatorial Theory Ser. A \textbf{20} (1976), no.~3, 292--299. \MR{429728}

\bibitem[Dek79]{fD79}
\bysame, \emph{Strongly nonrepetitive sequences and progression-free sets}, J.
  Combin. Theory Ser. A \textbf{27} (1979), no.~2, 181--185. \MR{542527}

\bibitem[EJS74]{EJS74}
R.~C. Entringer, D.~E. Jackson, and J.~A. Schatz, \emph{On nonrepetitive
  sequences}, J. Combinatorial Theory Ser. A \textbf{16} (1974), 159--164.
  \MR{332533}

\bibitem[EKW94]{EKW94}
A.~Eizenberg, Y.~Kifer, and B.~Weiss, \emph{Large deviations for {${\bf
  Z}^d$}-actions}, Comm. Math. Phys. \textbf{164} (1994), no.~3, 433--454.
  \MR{1291239}

\bibitem[FS95]{FS95}
Aviezri~S. Fraenkel and R.~Jamie Simpson, \emph{How many squares must a binary
  sequence contain?}, Electron. J. Combin. \textbf{2} (1995), Research Paper 2,
  approx. 9. \MR{1309124}

\bibitem[KLO16]{KLO16}
Dominik Kwietniak, Martha \L\c{a}cka, and Piotr Oprocha, \emph{A panorama of
  specification-like properties and their consequences}, Dynamics and numbers,
  Contemp. Math., vol. 669, Amer. Math. Soc., Providence, RI, 2016,
  pp.~155--186. \MR{3546668}

\bibitem[KR11]{KR11}
Roman Kolpakov and Micha\"el Rao, \emph{On the number of {D}ejean words over
  alphabets of 5, 6, 7, 8, 9 and 10 letters}, Theoret. Comput. Sci.
  \textbf{412} (2011), no.~46, 6507--6516. \MR{2883022}

\bibitem[Kri75]{wK75}
W.~Krieger, \emph{On generators in ergodic theory}, Proceedings of the
  {I}nternational {C}ongress of {M}athematicians ({V}ancouver, {B}.{C}., 1974),
  {V}ol. 2, Canad. Math. Congr., Montreal, QC, 1975, pp.~303--308. \MR{422576}

\bibitem[KS04]{KS04}
Juhani Karhum\"{a}ki and Jeffrey Shallit, \emph{Polynomial versus exponential
  growth in repetition-free binary words}, J. Combin. Theory Ser. A
  \textbf{105} (2004), no.~2, 335--347. \MR{2046086}

\bibitem[Led77]{fL77}
Fran\c{c}ois Ledrappier, \emph{Mesures d'\'{e}quilibre d'entropie
  compl\`etement positive}, Dynamical systems, {V}ol. {II}---{W}arsaw,
  Ast\'{e}risque, vol. No. 50, Soc. Math. France, Paris, 1977, pp.~251--272.
  \MR{486419}

\bibitem[Lot83]{mL83}
M.~Lothaire, \emph{Combinatorics on words}, Encyclopedia of Mathematics and its
  Applications, vol.~17, Addison-Wesley Publishing Co., Reading, MA, 1983, A
  collective work by Dominique Perrin, Jean Berstel, Christian Choffrut, Robert
  Cori, Dominique Foata, Jean Eric Pin, Guiseppe Pirillo, Christophe
  Reutenauer, Marcel-P. Sch\"utzenberger, Jacques Sakarovitch and Imre Simon,
  With a foreword by Roger Lyndon, Edited and with a preface by Perrin.
  \MR{675953}

\bibitem[Mil12]{jM12}
Joseph~S. Miller, \emph{Two notes on subshifts}, Proc. Amer. Math. Soc.
  \textbf{140} (2012), no.~5, 1617--1622. \MR{2869146}

\bibitem[MNC07]{MC07}
M.~Mohammad-Noori and James~D. Currie, \emph{Dejean's conjecture and {S}turmian
  words}, European J. Combin. \textbf{28} (2007), no.~3, 876--890. \MR{2300768}

\bibitem[MO92]{MO92}
Jean Moulin~Ollagnier, \emph{Proof of {D}ejean's conjecture for alphabets with
  {$5,6,7,8,9,10$} and {$11$} letters}, Theoret. Comput. Sci. \textbf{95}
  (1992), no.~2, 187--205. \MR{1156042}

\bibitem[Och06]{pO06}
Pascal Ochem, \emph{A generator of morphisms for infinite words}, Theor.
  Inform. Appl. \textbf{40} (2006), no.~3, 427--441. \MR{2269202}

\bibitem[Pan84]{jP84}
Jean-Jacques Pansiot, \emph{{\`A} propos d'une conjecture de {F}. {D}ejean sur
  les r\'ep\'etitions dans les mots}, Discrete Appl. Math. \textbf{7} (1984),
  no.~3, 297--311. \MR{736893}

\bibitem[Pav19]{rP19}
Ronnie Pavlov, \emph{On controlled specification and uniqueness of the
  equilibrium state in expansive systems}, Nonlinearity \textbf{32} (2019),
  no.~7, 2441--2466. \MR{3957218}

\bibitem[Pav20]{rP20}
\bysame, \emph{On entropy and intrinsic ergodicity of coded subshifts}, Proc.
  Amer. Math. Soc. \textbf{148} (2020), no.~11, 4717--4731. \MR{4143389}

\bibitem[Pav22]{rP22}
\bysame, \emph{On subshifts with slow forbidden word growth}, Ergodic Theory
  Dynam. Systems \textbf{42} (2022), no.~4, 1487--1516. \MR{4389800}

\bibitem[PS05]{PS05}
C.-E. Pfister and W.~G. Sullivan, \emph{Large deviations estimates for
  dynamical systems without the specification property. {A}pplications to the
  {$\beta$}-shifts}, Nonlinearity \textbf{18} (2005), no.~1, 237--261.
  \MR{2109476}

\bibitem[PYY22]{PYY22}
Maria~Jose Pacifico, Fan Yang, and Jiagang Yang, \emph{Existence and uniqueness
  of equilibrium states for systems with specification at a fixed scale: an
  improved {C}limenhaga-{T}hompson criterion}, Nonlinearity \textbf{35} (2022),
  no.~12, 5963--5992. \MR{4500887}

\bibitem[QS16]{QS16}
Anthony Quas and Terry Soo, \emph{Ergodic universality of some topological
  dynamical systems}, Trans. Amer. Math. Soc. \textbf{368} (2016), no.~6,
  4137--4170. \MR{3453367}

\bibitem[Rao11]{mR11}
Micha\"{e}l Rao, \emph{Last cases of {D}ejean's conjecture}, Theoret. Comput.
  Sci. \textbf{412} (2011), no.~27, 3010--3018. \MR{2830264}

\bibitem[Ros25]{mR25}
Matthieu Rosenfeld, \emph{Finding lower bounds on the growth and entropy of
  subshifts over countable groups}, Groups Geom. Dyn. \textbf{19} (2025),
  no.~3, 989--1011. \MR{4945537}

\bibitem[RS85]{RS85}
Antonio Restivo and Sergio Salemi, \emph{Overlap-free words on two symbols},
  Automata on infinite words ({L}e {M}ont-{D}ore, 1984), Lecture Notes in
  Comput. Sci., vol. 192, Springer, Berlin, 1985, pp.~198--206. \MR{814744}

\bibitem[RSW05]{RSW05}
Narad Rampersad, Jeffrey Shallit, and Ming-wei Wang, \emph{Avoiding large
  squares in infinite binary words}, Theoret. Comput. Sci. \textbf{339} (2005),
  no.~1, 19--34. \MR{2142071}

\bibitem[Shu08]{aS08}
Arseny~M. Shur, \emph{Comparing complexity functions of a language and its
  extendable part}, Theor. Inform. Appl. \textbf{42} (2008), no.~3, 647--655.
  \MR{2434040}

\bibitem[Shu09]{aS09}
\bysame, \emph{Two-sided bounds for the growth rates of power-free languages},
  Developments in language theory, Lecture Notes in Comput. Sci., vol. 5583,
  Springer, Berlin, 2009, pp.~466--477. \MR{2544724}

\bibitem[Shu11]{aS11}
\bysame, \emph{Growth properties of power-free languages}, Developments in
  language theory, Lecture Notes in Comput. Sci., vol. 6795, Springer,
  Heidelberg, 2011, pp.~28--43. \MR{2862712}

\bibitem[Shu12]{aS12}
\bysame, \emph{Growth properties of power-free languages}, Comput. Sci. Rev.
  \textbf{6} (2012), no.~5-6, 187--208 (eng).

\bibitem[SS19]{SS19}
Jeffrey Shallit and Arseny Shur, \emph{Subword complexity and power avoidance},
  Theoret. Comput. Sci. \textbf{792} (2019), 96--116. \MR{4012267}

\bibitem[TS12]{TS12}
Igor~N. Tunev and Arseny~M. Shur, \emph{On two stronger versions of {D}ejean's
  conjecture}, Mathematical foundations of computer science 2012, Lecture Notes
  in Comput. Sci., vol. 7464, Springer, Heidelberg, 2012, pp.~800--812.
  \MR{3030481}

\bibitem[Wal82]{pW82}
Peter Walters, \emph{An introduction to ergodic theory}, Graduate Texts in
  Mathematics, vol.~79, Springer-Verlag, New York-Berlin, 1982. \MR{648108}

\bibitem[Yam09]{kY09}
Kenichiro Yamamoto, \emph{On the weaker forms of the specification property and
  their applications}, Proc. Amer. Math. Soc. \textbf{137} (2009), no.~11,
  3807--3814. \MR{2529890}

\end{thebibliography}

\end{document}